\documentclass[12pt]{amsart}
\usepackage{amsthm}
\newtheorem{theorem}{Theorem}[section]
\newtheorem{lemma}[theorem]{Lemma}
\newtheorem{proposition}[theorem]{Proposition}
\theoremstyle{definition}
\newtheorem{definition}[theorem]{Definition}

\theoremstyle{remark}
\newtheorem{remark}[theorem]{Remark}
\newtheorem{corollary}[theorem]{Corollary}
\newtheorem{example}[theorem]{Example}
\counterwithin*{section}{part}

\usepackage{amssymb}
\usepackage{empheq}
\usepackage{stmaryrd}
\usepackage{enumerate}
\usepackage[shortlabels]{enumitem}
\usepackage{calc}
\usepackage{url}

\usepackage{mathabx}

\usepackage{xcolor}

\newcommand{\norm}[1]{\left\lVert#1\right\rVert}
\usepackage{mathtools}

\DeclarePairedDelimiter{\floor}{\lfloor}{\rfloor}

\newcommand{\seq}[1]{\{{#1}\}}

\newcommand{\bignorm}[1]{\left\lVert {#1} \right\rVert}
\def\EE{\mathbb{E}}
\def\PP{\mathbb{P}}
\def\NN{\mathbb{N}}

\def\QQ{\mathbb{Q}}
\def\RR{\mathbb{R}}

\def\BB{\mathbb{B}}

\newcommand{\WEPAomega}{\textsf{WE}\mbox{-}\textsf{PA}^\omega}

\newcommand{\QFAC}{\textsf{QF}\mbox{-}\textsf{AC}}

\newcommand{\DC}{\textsf{DC}}

\usepackage[left=2.0cm,%
                right=2.0cm,%
                top=2.5cm,%
                bottom=3.5cm,%
                headheight=12pt,%
                a4paper]{geometry}%
\begin{document}
\title[A finitary Kronecker's lemma and the Strong Law of Large Numbers]{A finitary Kronecker's lemma and large deviations in the Strong Law of Large Numbers on Banach spaces}
\author{Morenikeji Neri}
\date{\today}
\maketitle
\vspace*{-5mm}
\begin{center}
{\scriptsize 
Department of Computer Science, University of Bath\\
E-mail: mn728@bath.ac.uk}
\end{center}
\maketitle

\begin{abstract}
We explore the computational content of Kronecker's lemma via the proof-theoretic perspective of proof mining and utilise the resulting finitary variant of this fundamental result to provide new rates for the Strong Law of Large Numbers for random variables taking values in type $p$ Banach spaces, which in particular are very uniform in the sense that they do not depend on the distribution of the random variables. Furthermore, we provide computability-theoretic arguments to demonstrate the ineffectiveness of Kronecker's lemma and investigate the result from the perspective of Reverse Mathematics. In addition, we demonstrate how this ineffectiveness from Kronecker's lemma trickles down to the Strong Law of Large Numbers by providing a construction that shows that computable rates of convergence are not always possible. Lastly, we demonstrate how Kronecker's lemma falls under a class of deterministic formulas whose solution to their Dialectica interpretation satisfies a continuity property and how, for such formulas, one obtains an upgrade principle that allows one to lift computational interpretations of deterministic results to quantitative results for their probabilistic analogue. This result generalises the previous work of the author and Pischke.
\end{abstract}

\section{Introduction}\label{sec:intro}
The proof mining program, which evolved from Kreisel’s ‘unwinding’ program \cite{kreisel:51:proofinterpretation:part1,kreisel:52:proofinterpretation:part2} and was further developed by Kohlenbach and collaborators, has been successful in extracting computational content from results in analysis and optimisation. This success is largely due to foundational results, in particular, the so-called proof mining metatheorems (see, for example, \cite{gerhardy2008general,kohlenbach2005some,kohlenbach:08:book}), which demonstrate that significant portions of analysis can be formalised in systems allowing for program extraction through proof interpretations (typically modifications of G{\"o}del’s Dialectica interpretation \cite{godel1958bisher}).

At first glance, probability theory might seem beyond the reach of proof mining methods due to the level of comprehension required to define key concepts, such as countable unions. However, by applying the methodology of proof interpretations, Avigad and collaborators were able to extract the computational content for some key results in probability and measure theory \cite{Avigad-Dean-Rute:Dominated:12,AVIGAD-GERHARDY-TOWSNER:10:Ergodic}. It was realised that the tameness of these ad hoc case studies in probability and measure theory was largely due to their limited use of infinite unions. Based on this observation, the author and Pischke developed a formal system and corresponding metatheorem \cite{NeriPischke2023}, the first such for proof mining and probability theory. This article explores the computational content of the Laws of Large Numbers in light of the formal methodology introduced in \cite{NeriPischke2023}.

The Law of Large Numbers is an important concept in probability theory. It is an empirical fact that the observed outcomes of an experiment should get closer to the expected value as the sample size increases. This observation has been made mathematically rigorous by the likes of Bernoulli, Markov, and Khintchine \footnote{For a detailed historical and technical explanation of the Laws of Large Numbers, one can refer to Seneta's work \cite{Seneta13}.}. However, Kolmogorov \cite{Kol1930} is credited with the Strong Law of Large Numbers, which states:

\begin{theorem}[Strong Law of Large Numbers]\label{thrm:SLLN}
Suppose $X_0, X_1,\ldots$ are independent, identically distributed (iid) real-valued random variables with  $\EE(|X_0|) < \infty$. Then,
\begin{equation*}
       \frac{1}{n}\sum_{i=0}^nX_i\to \EE(X_0)
\end{equation*}
 almost surely, that is, with probability $1$.
\end{theorem}
 Throughout this article, we typically write $S_n := \sum_{i=0}^nX_i$ and assume $\EE(X_0) = 0$. 

 Variants of the above Strong Law of Large numbers commonly studied in the literature are concerned with weakening the identical distribution assumption in Theorem \ref{thrm:SLLN}, which, however, entails that other additional assumptions must be included if we still want to conclude  $\frac{S_n}{n}\to 0$ almost surely. One such variant, concerned with the case when the $\seq{X_n}$ are no longer identically distributed, is the following other classical result of Kolmogorov:

\begin{theorem}[Kolmogorov's Strong Law of Large Numbers]\label{thrm:KLLN}
    Suppose $\seq{X_n}$ is a sequence of independent real-valued random variables with $\EE(X_n) = 0$ for all $n \in \NN$ and 
\begin{equation*}
        \sum_{i=0}^\infty \frac{Var(X_n)}{n^2} < \infty.
\end{equation*}
Then $\frac{S_n}{n} \to 0$ almost surely.
\end{theorem}

Furthermore, this result can be generalised, leading to the seminal Law of Large Numbers due to Chung:

\begin{theorem}[Chung's Law of Large Numbers c.f.\ \cite{chung:47:LLN}]
\label{thrm:chung}
   Suppose $\seq{X_n}$ is a sequence of independent real-valued random variables with $\EE(X_n) = 0$ for all $n \in \NN$. For each $n \in \NN$, let $\phi_n : \RR^+ \to \RR^+$ be a function such that
   \begin{equation}
   \label{eqn:chung}
      \frac{\phi_n(t)}{t} \mbox{ and } \frac{t^2}{\phi_n(t)}
   \end{equation}
 are nondecreasing and assume
 $$ \sum_{i=0}^\infty \frac{\EE(\phi_n(|X_n|))}{\phi_n(n)} < \infty.$$
Then $\frac{S_n}{n} \to 0$ almost surely.
\end{theorem}
\begin{remark}
   If $\phi_n(t):= t^2$, we immediately obtain Kolmogorov's Strong Law of Large Numbers. 
\end{remark}
The proofs of these results both use a fundamental result about real numbers known as Kronecker's lemma:

\begin{theorem}[Kronecker's lemma]
\label{thrm:realkron}
Let $\seq{x_n}$ be a sequence of real numbers and $0< a_0 \le a_1 \le \cdots$ be such that $a_n \to \infty$. 
If $\sum_{i=0}^\infty x_i < \infty$, then 
    \begin{equation*}
        \frac{1}{a_n}\sum_{i=0}^n a_ix_i \to 0
    \end{equation*}
    as $n \to \infty$. 
\end{theorem}
Concretely, in the proofs of Theorems \ref{thrm:KLLN} and \ref{thrm:chung}, one first establishes that 
$$\sum_{i=0}^\infty \frac{X_n}{n} < \infty$$
almost surely. Then, for each $\omega \in \Omega$, we have that
$$\sum_{i=0}^\infty \frac{X_n(\omega)}{n} < \infty \text{ implies } \frac{1}{n}\sum_{i=0}^n X_k(\omega) \to 0$$
by Kronecker's lemma with $a_n = n$ and $x_i = X_i(\omega)/i$. Thus
$$1=\PP\left(\left\{\omega \in \Omega:\sum_{i=0}^\infty \frac{X_n(\omega)}{n} < \infty\right\}\right) \le \PP\left(\left\{\omega \in \Omega:\frac{1}{n}\sum_{i=0}^n X_k(\omega) \to 0\right\}\right). $$

One can recognise this approach  as an application of the following (similarly proven) probabilistic analogue to Kronecker's lemma:

\begin{theorem}[Probabilistic Kronecker's lemma]
\label{thrm:realprobkron}
    Let $\seq{X_n}$ be a sequence of real-valued random variables and $0< a_0 \le a_1 \le \cdots$ be such that $a_n \to \infty$. 
If $\sum_{i=0}^\infty X_i < \infty$ almost surely, then 
    \begin{equation*}
        \frac{1}{a_n}\sum_{i=0}^n a_iX_i \to 0
    \end{equation*}
    almost surely.
\end{theorem}

In this article, we study Kronecker's lemma from a proof-theoretic perspective. We prove a quantitative finitary version of Kronecker's lemma for sequences in arbitrary normed spaces by solving its Dialectica interpretation. In addition, we demonstrate the ineffectiveness of Kronecker's lemma in that computable rates are not always present for the conclusion of the result. Furthermore, we prove that the existence of rates for Kronecker's lemma is equivalent to $\mathrm{ACA_0}$ when viewed from the perspective of Reverse Mathematics \cite{simpson2009subsystems}.

We shall then use our proof-theoretic investigation of Kronecker's lemma to study its probabilistic analogue and the Laws of Large Numbers. We shall use the quantitative version of Kronecker's lemma, which we obtained, to give a computational interpretation of its probabilistic analogue. Furthermore, we are able to generalise our strategy for obtaining a quantitative version of the probabilistic analogue of Kronecker's lemma. This will result in a general proof-theoretic transfer principle for obtaining quantitative versions of probabilistic results from their deterministic analogues that generalise \cite[Theorem 10.2]{NeriPischke2023}. Lastly, we use our quantitative results to give a computational interpretation to the generalisation of Theorem \ref{thrm:chung} to Banach spaces given in \cite{woyczynski1974random}, and we demonstrate the ineffectiveness of this result by showing that computable rates do not always exist.

\subsection{Related work}
\label{subsec:Related}
Quantitatively studying the Laws of Large Numbers is an active area of research. Several recent articles have been on this topic, including \cite{siegmund1975large,kachurovskii2019measuring,korchevsky2018rate,luzia_2018}. However, the work by the author in \cite{Ner2023:LLN} is the only research effort to achieve such results using the proof mining methodology.

Moreover, there has been a recent resurgence in applying proof theory to probability theory. In collaboration with Powell \cite{NeriPowell:martingales:2024}, the author has built on the work of Avigad, Gerhardy and Towsner \cite{AVIGAD-GERHARDY-TOWSNER:10:Ergodic} to obtain various quantitative results about stochastic processes, including Doob's martingale convergence theorem. Further developments have also been made in the applications of proof theory to the study of stochastic processes; we highlight another collaboration of the author and Powell \cite{NeriPowell:RS:2024}, where they provide a quantitative version of the celebrated Robbins-Siegmund theorem \cite{robbins1971convergence} and the recent work of Pischke and Powell \cite{PischkePowell:Halpern:2024} on obtaining quantitative results concerning a newly introduced stochastic Halpern scheme.

Finally, we mention the recent articles \cite{freund2022re,freund2023bounds} and \cite{arthan2021borel}. Although they are not as directly related as the previously mentioned works,\cite{freund2022re,freund2023bounds} represents the only applications of the proof mining methodology in extracting quantitative data from results concerning so-called type $p$ Banach spaces (which will be introduced in the following section) and \cite{arthan2021borel} is the first application of this methodology in probability theory after that of Avigad and his collaborators.
\subsection{An outline for the rest of the paper}
\label{subsec:paperoutline}
We start, in Section \ref{sec:prelim}, by introducing the notions of quantitative deterministic converges and probabilistic convergence we shall need in this paper. In addition, we give a brief introduction to probability theory on Banach spaces.

In Section \ref{sec:Kron}, we introduce Kronecker's lemma on arbitrary normed spaces. We first prove a finitary quantitative version of this result and use it to extract metastable rates for the conclusion in terms of metastable rates from the premise. Furthermore, we justify our metastable solution to Kronecker's lemma by demonstrating that computable rates do not exist in general. We further also investigate the existence of rates for Kronecker's lemma from the perspective of Reverse Mathematics \cite{simpson2009subsystems}.

Then, in Section \ref{sec:chung}, We first use our quantitative Kronecker's lemma to obtain rates for its probabilistic analogue. Our quantitative version of the probabilistic Kronecker's lemma is then used to give a computational interpretation to the generalisation of Chung's Law of Large Numbers on type $p$ Banach spaces (c.f.\ Section \ref{subsec:probonbanach}), given in \cite{woyczynski1974random}.

We then investigate Kronecker's lemma in the light of the formal system introduced in \cite{NeriPischke2023}. In particular, we discuss how our quantitative version of Kronecker's lemma satisfies the criterion presented in \cite{NeriPischke2023} needed to obtain a quantitative version of its probabilistic analogue, and we generalise this transfer result.

Lastly, in Section \ref{sec:computability}, we provide a construction demonstrating that computable rates of almost sure convergence (c.f.\ Section \ref{subsec:quantconv}) do not exist in general for Kolmogorov's Strong Law of Large Numbers, demonstrating the ineffectiveness of this result.
\section{Mathematical preliminaries}
\label{sec:prelim}
\subsection{Probability on Banach spaces}
\label{subsec:probonbanach}
Throughout this article, fix a normed space $(\BB, \norm{\cdot})$ and a probability space $(\Omega, \mathcal{F}, \PP)$. We summarise the relevant parts of the theory of random variables taking values in a Banach space. We shall mainly follow the treatment given in \cite{ledoux2013probability}.\\
If $\BB$ is a Banach space, then the natural definition one would give a Borel random variable taking values in $\BB$ is a measurable map from $(\Omega, \mathcal{F}, \PP)$ to $\BB$ endowed with the Borel sigma-algebra generated by its open sets. However, as noted in \cite{ledoux2013probability}, this definition is too general to develop a useful theory of probability (for example, the set of such random variables does not form a vector space. It is not even the case that this class is closed with respect to addition c.f.\ \cite{Masani76}). Therefore, it is standard to assume random variables $X$ are tight (sometimes referred to as Radon), that is,
for all $\varepsilon > 0$ there exists a compact set $K \subseteq \BB$ such that 
\begin{equation*}
    \PP(X \in K)\ge 1 - \varepsilon.
\end{equation*}

Denote the set of tight Borel random variables in $\BB$ by $L_0$. Working with such random variables ensures we can add and multiply them by scalars without worrying about measurability. Furthermore, we also have the set $L_p:= \{X \in L_0\, \vert\, \EE(\norm{X}^p) < \infty\}$ is also a vector space. Lastly, we note that a random variable is tight if and only if it takes values on a separable subset of $\BB$ (c.f.\ \cite[Section 2.1]{ledoux2013probability}), and so it is standard to assume $\BB$ is separable. By random variables, we always refer to elements of $L_0$.

We also introduce the notion of integration for random variables taking values in $\BB$, attributed to Bochner, via the following theorem:
\begin{theorem}[c.f.\ Theorem II.11 of \cite{li2017introduction}]
\label{thrm:bochner}
    There exists a unique linear mapping $\EE : L_1(\BB) \to \BB$ called the expectation such that:
    \begin{itemize}
        \item [(a)] $\EE(X) = \sum_{i=0}^n \PP(A_i)x_i$ for all $X = \sum_{i=0}^n 1_{A_i}x_i$ with $\seq{A_i} \subseteq \mathcal{F}$, and $\seq{x_i} \subseteq \BB$
        \item [(b)] $\norm{\EE(X)} \le \EE(\norm{X})$ for all $X \in L_1(\BB)$.
    \end{itemize}
\end{theorem}
For $1\le p\le 2$, $\BB$ is said to have (Rademacher) type $p$, if there exists a constant $B$ such that, for every independent sequence of random variables $\seq{\varepsilon_n}$ satisfying 
\begin{equation*}
    \PP(\varepsilon_n = \pm 1) = \frac{1}{2}
\end{equation*}
(such a sequence is sometimes known as a Rademacher sequence) and sequence of element $\seq{x_n}$ in $\BB$, we have 
$$\EE\left(\norm{\sum_{i=0}^nx_i\varepsilon_i}^p\right) \le B\sum_{i=0}^n\norm{x_i}^p.$$
By the triangle inequality, every Banach space is type $1$. Furthermore, if $\BB$ is type $p$, then it is of type $p'$ for all $p' \le p$ (c.f.\ \cite[Proposition 9.12]{ledoux2013probability}) and a Banach space is of type $2$ if and only if it is isomorphic to a Hilbert space. 

In \cite{hoffmann1976law}, it is shown that for $1\le p \le 2$, if $\seq{X_n}$ is an iid sequence of random variables taking values in $\BB$, with $\EE(X_n) = 0$ (where $\EE$ is the Bochner integral introduced in Theorem \ref{thrm:bochner}) and
    $$ \sum_{n=0}^\infty \frac{\EE(\norm{X_n}^p)}{n^p}< \infty,$$ 
    $\BB$ being a type $p$ Banach space is both a necessary and sufficient condition for the conclusion 
    \begin{equation*}
        \frac{S_n}{n} \to 0
    \end{equation*}
    almost surely to hold.

In \cite{woyczynski1974random}, Woyczynski shows that the analogy of Theorem \ref{thrm:chung}, where $(\ref{eqn:chung})$ is replaced with 
   \begin{equation*}
      \frac{\phi_n(t)}{t} \mbox{ and } \frac{t^p}{\phi_n(t)}
   \end{equation*}
being nondecreasing, holds in type $p$ Banach spaces. What Woyczynski actually showed, in \cite{woyczynski1974random}, was that this result holds in spaces such that there exists a constant $C$ satisfying,
\begin{equation}
    \label{eqn:Woy}
    \EE\left(\norm{\sum_{i=0}^nY_i}^p\right) \le C\sum_{i=0}^n\EE(\norm{Y_i}^p)
\end{equation}
 for all independent random variables taking values in $\BB$, $Y_0, \cdots, Y_n$ , with $0$ expected value and finite $pth$ moment \footnote{Woyczynski was working in so-called $\mathcal{G}_\alpha$ which are type $(\alpha -1)$ spaces but they are smoother than general type $(\alpha-1)$ spaces. However, they did not use further properties of such spaces other than the fact that the relation (\ref{eqn:Woy}) was satisfied. So their result does indeed hold in general type $(\alpha-1)$ spaces. This observation was noted in \cite{hoffmann1976law}.}. If $\BB$ is a type $p$ Banach space with constant $B$, then (\ref{eqn:Woy}) holds with $W = (2B)^p$ (c.f.\cite[Proposition 9.11]{ledoux2013probability}). Therefore $\BB$ is type $p$ if and only if $(\ref{eqn:Woy})$ holds. We give a quantitative version of Woyczynski's result in Section \ref{sec:chung}.  
\subsection{Quantitative Convergence}
\label{subsec:quantconv}
Suppose $\seq{x_n}$ is a Cauchy sequence of elements in $\BB$. We say the function $r$ is a rate of (Cauchy) convergence for $\seq{x_n}$ if,
\begin{equation*}
    \forall \varepsilon\in\QQ_+\, \forall n,m \ge r(\varepsilon) \left(\norm{x_n - x_m} < \varepsilon\right).
\end{equation*}
It is known that computable rates of convergences do not exist in general, even if the sequence in question is computable \cite{specker:49:sequence}. A rate of convergence is a direct computational interpretation of the definition of Cauchy convergence. When we apply proof interpretations to statements such as Cauchy convergence, the process results in a formula equivalent, over classical logic, to the original statement with a different computational challenge that one can hope to solve. Observe that one can formulate Cauchy convergence as
\begin{equation}
\label{eqn:newCauchy}
    \forall \varepsilon\in\QQ_+\, \exists N \, \forall k \, \forall n,m \in [N;N+k] \left( \norm{x_n - x_m} < \varepsilon\right)
\end{equation}
where $[a;b]:=\{a,a+1,\ldots,b\}$ if $a \le b$ and empty otherwise. Consider the following reformulation of Cauchy convergence:
\begin{equation}
\label{eqn:metastable:Cauchy}
\forall \varepsilon\in\QQ_+\, \forall g:\NN\to \NN\,  \exists N\, \forall n,m\in [N;N+g(N)]( \norm{x_n - x_m}< \varepsilon). 
\end{equation}
This formulation changes the direct computational challenge of Cauchy convergence (that is, a rate of convergence) to finding a functional $\Phi:\QQ_+\times (\NN\to\NN)\to\NN$ bounding $N$ in (\ref{eqn:metastable:Cauchy}) i.e.
\begin{equation}
\label{eqn:ROM}
\forall \varepsilon\in\QQ_+\, \forall g:\NN\to \NN\,  \exists N\leq \Phi(\varepsilon,g)\, \forall n,m\in [N;N+g(N)]( \norm{x_n - x_m}< \varepsilon). 
\end{equation}
Such a functional is known as a rate of (Cauchy) metastability, and the systematic extraction of such rates, using proof-theoretic techniques, are standard results in proof mining with recent results including \cite{freund2023bounds,NeriPowell2023, powell2023finitization}. The idea of metastability was rediscovered in mainstream mathematics by Tao \cite{tao:07:softanalysis,tao:08:ergodic}, who was interested in finitizations of infinitary notions in mathematics and found nontrivial applications in number theory and ergodic theory. 
\begin{remark}
    Going from (\ref{eqn:newCauchy}) to (\ref{eqn:metastable:Cauchy}) can be seen formally as an application of the negative translation in combination with the Dialectica interpretation (see \cite{kohlenbach:08:book} for details) and thus the equivalence of both formulations in a suitably constructive theory follows. However, a proof of equivalence is readily obtainable directly:
    
     (\ref{eqn:metastable:Cauchy}) follows from (\ref{eqn:newCauchy}) since if $N$ is such that $\norm{x_n-x_m}<\varepsilon$ for all $n,m\geq [N; N+k]$, for all $k \in \NN$, then we may set $k = g(N)$ and obtain (\ref{eqn:metastable:Cauchy}). For the other direction, we argue by contradiction. If (\ref{eqn:newCauchy}) does not hold then there exists some $\varepsilon \in \QQ^+$ such that
\begin{equation*}
\forall N\, \exists\, k\,\exists n,m \in [N,N+k] (\norm{x_n-x_m}\ge\varepsilon)
\end{equation*}
and therefore (by the axiom of choice), there exists some function $g:\NN\to\NN$ satisfying
\begin{equation*}
\forall N\, \exists n,m\in [N,N+g(N)](\norm{x_n-x_m}\ge\varepsilon)
\end{equation*}
which contradicts (\ref{eqn:metastable:Cauchy}).
\end{remark}

If $\seq{Y_n}$ is a sequence of random variables, then $Y_n$ converges almost surely simply means
\begin{equation*}
    \mathbb{P}(\seq{Y_n} \text{ converges}) = 1.
\end{equation*}
Observe that the set 
$$\{\omega \in \Omega: \seq{Y_n(\omega)} \text{ converges}\} = \bigcap_{k = 0}^\infty\bigcup_{N = 0}^\infty\bigcap_{i = N}^\infty\bigcap_{j = N}^\infty \{\omega \in \Omega: \norm{Y_i - Y_j} < 2^{-k}\} $$
and since $L_0$ forms a vector space, $Y_i - Y_j$ is a random variable for each $i, j$. Thus, the above set is measurable.\\

It is not obvious how one can give a direct computational interpretation to this definition. However, it can be shown that almost sure convergence is equivalent to almost uniform convergence,
\begin{equation}
\label{eqn:a.u}
  \forall \varepsilon, \lambda \in \QQ^+\, \exists N \in \NN\, \forall k \in \NN\, \left(\mathbb{P}\left(\max_{N\le n \le N+ k}\norm{Y_n -  Y_N} \ge \varepsilon\right)< \lambda\right).
\end{equation}
This result is known as Egorov’s theorem, and the above formulation of almost sure convergence has a clear computational challenge (a function taking $\varepsilon$ and $\lambda$ realising $N$). We call a solution a rate of almost sure (Cauchy) convergence. Furthermore, we can also obtain a metastable notion of almost uniform Cauchy convergence. 
\begin{equation*}
    \forall \varepsilon, \lambda \in \QQ^+\, \forall K:\NN\to \NN\, \exists N \,\left(\mathbb{P}\left(\max_{N \le n \le N+K(N)}\norm{Y_N -  Y_n} \ge \varepsilon\right)< \lambda\right).
\end{equation*}
A solution to the computational challenge (a functional bounding $N$) is known as a rate of almost sure metastable (Cauchy) convergence. Inspired by the idea of metastability for sequences of real numbers, this notion of a rate originated in \cite{Avigad-Dean-Rute:Dominated:12,AVIGAD-GERHARDY-TOWSNER:10:Ergodic}. However, almost sure metastability can be seen, formally, as an application of the double negation translation in combination with the Dialectica translation of a suitable formulation of almost sure Cauchy convergence in the system introduced in \cite{NeriPischke2023}. 

We also have analogous definitions for convergence to a known limit (for example, we could have a rate of convergence to $0$ or a rate of metastable almost sure convergence to a fixed random variable). 

It is a folklore result in applied proof theory that a rate of convergence can be seen as a rate of metastable convergences. More precisely, 

\begin{theorem}[Folklore]
\label{thrm:rocimpram}
   $\Phi : \QQ^+ \to \NN$ is a rate of convergence for a sequence $\seq{x_n}$ iff  $\Phi^M$ defined as $\Phi^M(\varepsilon,g) = \Phi(\varepsilon)$, for all $\varepsilon \in \QQ^+, g:\NN \to \NN$, is a rate of metastability for $\seq{x_n}$.
\end{theorem}

\begin{proof}
    For the forward direction, let $\varepsilon \in \QQ^+$ and $g:\NN \to \NN$ be given. Then, taking  $N = \Phi^M(\varepsilon,g) = \Phi(\varepsilon)$, we have (from the fact that  $\Phi$ is a rate of convergence) $\forall n,m \ge N\,(|x_n - x_m| < \varepsilon)$. So, in particular $\forall n,m \in [N;N+g(N)](|x_n - x_m| < \varepsilon)$.

    For the converse, let $\varepsilon \in \QQ^+$ be given. Take $p,q \ge \Phi(\varepsilon)$. Define $g:\NN \to \NN$ as, $g(n) = \max\{p,q\}$. Since $\Phi^M$ is a rate of metastability, there exists $N \le \Phi^M(\varepsilon,g) = \Phi(\varepsilon) \le p, q$ such that $\forall n,m \in [N; N +  \max\{p,q\}] (|x_n - x_m| < \varepsilon)$. Since it is clear that both $p, q \in [N; N +  \max\{p,q\}]$ we have $|x_p - x_q| < \varepsilon$. 
\end{proof}

Through a simple adaptation of the above proof, we obtain a similar result for rates to known limits.

A similar result holds in the probabilistic case when considering convergence to a limit.

\begin{theorem}
\label{thrm:probrocimpram}
   $\Phi: \QQ^+\times\QQ^+ \to \NN$ is a rate of almost sure convergence for a sequence of random variables $\seq{Y_n}$ to a random variable $Y$ iff  $\Phi^M$ defined as $\Phi^M(\varepsilon,\lambda,K) = \Phi(\varepsilon,\lambda)$, for all $\varepsilon,\lambda \in \QQ^+, K:\NN \to \NN$, is a rate of metastable almost sure convergence for $\seq{Y_n}$ to $Y$.
\end{theorem}

\begin{proof}
    For the forward direction, let $\varepsilon,\lambda \in \QQ^+$ and $K:\NN \to \NN$ be given. Then taking  $N = \Phi^M(\varepsilon,\lambda,K) = \Phi(\varepsilon,\lambda)$, we have for $k := N+K(N)$ 
    \begin{equation*}
        \mathbb{P}\left(\max_{N \le n\le N+K(N)}\norm{Y_n -  Y} \ge \varepsilon\right) = \mathbb{P}\left(\max_{N \le n\le k}\norm{Y_n -  Y} \ge \varepsilon\right)< \lambda.
    \end{equation*}
    For the converse, let $\varepsilon,\lambda \in \QQ^+$ be given. Let $N := \Phi(\varepsilon,\lambda)$. For each $p \in \NN$, define $K_p:\NN \to \NN$ as, $K_p(n) = N+p$. Since $\Phi^M$ is a rate of metastability, there exists $N_p \le \Phi^M(\varepsilon,\lambda,K_p) = \Phi(\varepsilon,\lambda) = N$ such that 
      \begin{equation*}
        \mathbb{P}\left(\max_{N_p \le n \le N_p+N+p}\norm{Y_n -  Y} \ge \varepsilon\right) < \lambda.
    \end{equation*}
    Thus, we have for all $p \in \NN$,
    \begin{equation*}
        \mathbb{P}\left(\max_{N \le n \le N+p}\norm{Y_n -  Y} \ge \varepsilon\right) \le\mathbb{P}\left(\max_{N_p \le n \le N_p+N+p}\norm{Y_n -  Y} \ge \varepsilon\right)< \lambda.
    \end{equation*}
and we are done.
\end{proof}

\begin{remark}
\text{}
    \label{rem:rocimpram}
 The above shows that if a rate of metastable (almost sure to a limit) convergence is independent of its function part, it can be regarded as a rate of (almost sure to a limit) convergence.
 
 The results we study in this paper demonstrate the (almost sure) convergence of sequences of elements of $\BB$ (random variables taking values in $\BB$) given assumptions about sequences converging. Thus, our quantitative versions of these results will be rates of metastable (almost sure) convergences in terms of rates of metastable (almost sure) convergences coming from our assumptions. All the metastable rates will be uniform in the sense that if we assume we have rates of (almost sure) convergences in the assumptions, the metastable rates of (almost sure) convergence we obtain for the conclusions will be independent of their function part, so can be regarded as rates of convergences.
\end{remark}
\section{The computational content of Kronecker's lemma }
\label{sec:Kron}
This section aims to study the computational content of Kronecker's lemma. It turns out that the direct computational interpretation that one can give to Kronecker's lemma, that is, obtaining rates for the conclusion in terms of rates from the premise, is too weak to be able to get a computational interpretation of the probabilistic Kronecker's lemma. We need a stronger result, which can be seen as a finitary quantitative formulation of Kronecker's lemma similar to results in \cite{powell2023finitization,powell2020note}. This result will also tell us information about what error in the premise is required to produce the error we want in the conclusion.
\subsection{A quantitative version of Kronecker's lemma }
\label{subsec:quantkron}
We start by giving a proof of Kronecker's lemma for sequences of elements in $\BB$. The proof we present is a fleshed-out version of the proof of Theorem $A.6.2$ in \cite{gut2005probability}.

\begin{theorem}[Kronecker's lemma on $\BB$]
\label{thrm:banachkron}
Let $\seq{x_n}$ be a sequence of elements in $\BB$ and $0< a_0 \le a_1 \le \cdots$ be such that $a_n \to \infty$. 
If $\{\sum_{i=0}^n x_i\}$ is Cauchy, then 
    \begin{equation*}
        \frac{1}{a_n}\sum_{i=0}^n a_ix_i \to 0
    \end{equation*}
    as $n \to \infty$. 
\end{theorem}
\begin{proof}
Let $\varepsilon > 0$ be given. Define $s_n = \sum_{i=0}^n x_i$, by our hypothesis, $s_n$ is Cauchy. Take $M \in \NN$ such that 
\begin{equation*}
\norm{s_n - s_M} < \frac{\varepsilon}{4}
\end{equation*}
for all $n \ge M$.
 We first observe, by summation by parts, that for all  $n \ge M$,
\begin{equation*}
\bignorm{\frac{1}{a_n}\sum_{i=0}^n a_ix_i}  = \bignorm{s_n - \frac{1}{a_n}\sum_{i=0}^{n-1} (a_{i+1}-a_i)s_i} 
\end{equation*}
the right-hand side of the above becomes,
\begin{equation*}
     \begin{aligned}
           &\bignorm{s_n - \frac{1}{a_n}\sum_{i=0}^{M-1} (a_{i+1}-a_i)s_i - \frac{1}{a_n}\sum_{i=M}^{n-1} (a_{i+1}-a_i)s_M - \frac{1}{a_n}\sum_{i=M}^{n-1} (a_{i+1}-a_i)(s_i-s_M)} \\
           &\le \bignorm{s_n -(1-\frac{a_M}{a_n})s_M} + \bignorm{\frac{1}{a_n}\sum_{i=0}^{M-1} (a_{i+1}-a_i)s_i} + \bignorm{\frac{1}{a_n}\sum_{i=M}^{n-1} (a_{i+1}-a_i)(s_i-s_M)} \\
           &\le \norm{s_n - s_M} +\bignorm{\frac{a_Ms_M}{a_n}}+ \bignorm{\frac{1}{a_n}\sum_{i=0}^{M-1} (a_{i+1}-a_i)s_i} + \frac{1}{a_n}\sum_{i=M}^{n-1} (a_{i+1}-a_i)\norm{s_i-s_M} \\
           & <\frac{\varepsilon}{4} +\bignorm{\frac{a_Ms_M}{a_n}}+ \bignorm{\frac{1}{a_n}\sum_{i=0}^{M-1} (a_{i+1}-a_i)s_i} + \frac{\varepsilon}{4}\frac{1}{a_n}\sum_{i=M}^{n-1} (a_{i+1}-a_i)\\
           &\le\frac{\varepsilon}{2} +\bignorm{\frac{a_Ms_M}{a_n}}+ \bignorm{\frac{1}{a_n}\sum_{i=0}^{M-1} (a_{i+1}-a_i)s_i}.
     \end{aligned}
\end{equation*}
Now since $M$ is fixed and  $\seq{a_n}$ is an increasing sequence that tends to infinity, we can take $n$ large enough to ensure $\norm{\frac{a_Ms_M}{a_n}}$ and $\norm{\frac{1}{a_n}\sum_{i=0}^{M-1} (a_{i+1}-a_i)s_i}$ are both $< \frac{\varepsilon}{4}$.
\end{proof}

We now prove our finitary quantitative Kronecker's lemma.  
\begin{theorem}[Finitary Kronecker's lemma ]
    \label{thrm:quantkron}
Let $\seq{x_n}$ be a sequence of elements in $\BB$ and $0< a_0 \le a_1 \le \cdots$. For each $n \in \NN$ and $x \ge 0$, define, $s_n := \sum_{i=0}^n x_i$ and $f_{\seq{a_n}}(x) := \min\{n \in \NN : a_n \ge x\}$. 

     Now for every function $\gamma:\QQ^+ \to \NN$, sequence of natural numbers $\seq{z_n}$, $\varepsilon \in \QQ^+$ and  $w \in \NN$, if $M = \gamma(\frac{\varepsilon}{4})$ satisfies,
     \begin{equation}
     \label{eqn:con:maj}
         \forall i \le M \, (z_M \ge \norm{s_i})
     \end{equation}
    and
    \begin{equation}
    \label{eqn:kron:proof:1}
        \norm{s_n - s_M} < \frac{\varepsilon}{4}
    \end{equation}
    for all $n \in [M,w]$, then  $N = \Gamma_{\seq{a_n}}(\gamma, \seq{z_n},\varepsilon)$ satisfies,
    \begin{equation*}
        \bignorm{\frac{1}{a_n}\sum_{i=0}^n a_ix_i} < \varepsilon
    \end{equation*}
    for all $n \in [N, w]$. Where 
    $$\Gamma_{\seq{a_n}}(\gamma, \seq{z_n},\varepsilon) := \max\left\{\gamma\left(\frac{\varepsilon}{4}\right),f_{\seq{a_n}}\left(\frac{4a_{\gamma(\frac{\varepsilon}{4})}z_{\gamma(\frac{\varepsilon}{4})}}{\varepsilon}\right)\right\}.$$
\end{theorem}

\begin{proof}
 We have for each $n \in [N, w]$, 
\begin{equation*}
    \begin{aligned}
        \bignorm{\frac{1}{a_n}\sum_{i=0}^n a_ix_i}&\le \norm{s_n - s_M} +\bignorm{\frac{a_Ms_M}{a_n}}+ \bignorm{\frac{1}{a_n}\sum_{i=0}^{M-1} (a_{i+1}-a_i)s_i} + \frac{1}{a_n}\sum_{i=M}^{n-1} (a_{i+1}-a_i)\norm{s_i-s_M}\\
       &< \frac{\varepsilon}{4} +\bignorm{\frac{a_Ms_M}{a_n}}+ \bignorm{\frac{1}{a_n}\sum_{i=0}^{M-1} (a_{i+1}-a_i)s_i}+\frac{\varepsilon}{4}\frac{1}{a_n}\sum_{i=M}^{n-1} (a_{i+1}-a_i)\\
        &\le \frac{\varepsilon}{2} + \frac{a_Mz_M}{a_n} + \frac{a_Mz_M}{a_n} \le \varepsilon.
    \end{aligned}
\end{equation*}
The first line follows from precisely the first three lines in the calculation in the proof of Theorem \ref{thrm:banachkron}. To get the second line, we use (\ref{eqn:kron:proof:1}) to bound the first term and the fact that $[M, n-1] \subseteq [N,w]$ (since $M\ge N$ and $n \le w$) and (\ref{eqn:kron:proof:1}) allows us to bound the last term. The third line follows from the bounding condition of $\seq{z_n}$ on $\seq{s_n}$ and simplification.
\end{proof}
We can now obtain a quantitative version of Kronecker's lemma, which translates rates from the premise to rates for the conclusion.
\begin{corollary}
\label{cor:quantkron}
   Let $\seq{x_n}$, $\seq{a_n}$, $\seq{s_n}$, $f_{\seq{a_n}}$ be as in Theorem \ref{thrm:quantkron} and let $\seq{z_n}$ be a sequence of nondecreasing natural numbers satisfying $z_n \ge \norm{s_n}$ for all $n$. 

 Suppose  $\sum_{i = 1}^n x_i$ is Cauchy with rate of metastability $\Phi$. Then 
 $$\frac{1}{a_n}\sum_{i=0}^n a_ix_i$$ converges to $0$ with rate of metastability
    \begin{equation*}
        \kappa_{\Phi,\seq{a_n},\seq{z_n}}(\varepsilon,g) = \max\left\{Q,f_{\seq{a_n}}\left(\frac{4a_Qz_Q}{\varepsilon}\right)\right\}
    \end{equation*}
    Where, $Q := \Phi(\frac{\varepsilon}{4},h_{\varepsilon,g,\seq{a_n},\seq{z_n}})$ and 
    \begin{equation*}
            h_{\varepsilon, g,\seq{a_n},\seq{z_n}}(n) = \tilde{g}\left(\max\left\{n,f_{\seq{a_n}}\left(\frac{4a_nz_n}{\varepsilon}\right)\right\}\right)\\
    \end{equation*}
    with $\tilde{g}(n)= n+g(n)$.
\end{corollary}
\begin{proof}
Let $\varepsilon > 0, g:\NN \to \NN$ be given. By definition $\exists M \le Q = \Phi(\frac{\varepsilon}{4}, h_{\varepsilon,g,\seq{a_n},\seq{z_n}})$ such that
     \begin{equation*}
        |s_n - s_M| < \frac{\varepsilon}{4}
    \end{equation*}
    for all $n \in [M,h_{\varepsilon,g,\seq{a_n},\seq{z_n}}(M)] \subseteq [M,M+ h_{\varepsilon,g,\seq{a_n},\seq{z_n}}(M)]$. Now letting
    $\gamma(\sigma) = M$, for all $\sigma \in \QQ^+$ gives, by Theorem \ref{thrm:quantkron} ((\ref{eqn:con:maj}) follows since $\seq{z_n}$ is nondecreasing abound bounds $\seq{\norm{s_n}}$ respectively), 
    $$N = \Gamma_{\seq{a_n}}(\gamma, \seq{z_n},\varepsilon) = \max\left\{M,f_{\seq{a_n}}\left(\frac{4a_Mz_M}{\varepsilon}\right)\right\} \le \max\left\{Q,f_{\seq{a_n}}\left(\frac{4a_Qz_Q}{\varepsilon}\right)\right\},$$ 
    (the last inequality follows since $M \le Q$ and $f_{\seq{a_n}}, \seq{a_n},\seq{z_n}$ are all non-decreasing) and $w =  h_{\varepsilon,g,\seq{a_n},\seq{z_n}}(M) = N+ g(N)$ satisfies   
        \begin{equation*}
        \bignorm{\frac{1}{a_n}\sum_{i=0}^n a_ix_i} < \varepsilon
    \end{equation*}
    for all $n \in [N, N+g(N)]$, so we are done.  
\end{proof}
\begin{remark}
In light of Remark \ref{rem:rocimpram}, if $\Phi$ above is a rate of convergence, then we get a rate of convergence to zero given by the above expression, but with 
\begin{equation*}
    Q= \Phi(\varepsilon/4).
\end{equation*}
\end{remark}
\begin{remark}
    In both Theorem \ref{thrm:quantkron} and Corollary \ref{cor:quantkron} we can replace $f_{\seq{a_n}}$ by any nondecreasing function $f^*$ bounding $f_{\seq{a_n}}$, that is, satisfying $f^*(x)\ge f_{\seq{a_n}}(x)$ for all $x\ge 0$.
\end{remark}

\subsection{Computability of rates and the Reverse Mathematics of Kronecker's lemma}
\label{subsec:kroncomp}
We mentioned in Section \ref{subsec:quantconv} that there exist sequences of converging rational numbers that do not converge with a computable rate of convergence. The most famous of this construction is due to Specker in \cite{specker:49:sequence}, where he constructed a bounded monotone sequence of rational numbers that converge without a computable rate of Cauchy convergence, thus demonstrating that general rates of convergences cannot be extracted from any proof of the monotone convergence principle. A modification of Specker's construction yields a similar result to Kronecker's lemma. 
\begin{example}
\label{thrm:kroncomp}
    We can construct a sequence of rational numbers $\seq{x_n}$ such that $\sum_{i=0}^\infty x_i$ converges, but $\frac{1}{n+1}\sum_{i=0}^n (i+1)x_i$ converges to zero, without a computable rate of convergence.
    
    Let A be a recursively enumerable set that is not recursive (e.g. the halting set). Let $\seq{a_n}$ be a recursive enumeration of the elements in A. Let $x_i = 2^{-a_i}$. So $\seq{x_n}$ is a positive sequence of rational numbers and we have, 
    \begin{equation*}
        \sum_{i=0}^\infty x_i \le \sum_{i=0}^\infty 2^{-i} = 1
    \end{equation*}
    Now suppose, for contradiction, $\frac{1}{n}\sum_{i=0}^n ix_i$ converges to 0 with a computable rate of convergence $\phi$. We describe an effective procedure to determine whether $k \in A$, for all $k \in \NN$. Suppose, $k = a_n$ with $n > \phi(2^{-k})$, then 
    \begin{equation*}
        \frac{1}{n+1}\sum_{i=0}^n (i+1)x_i \le 2^{-k}
    \end{equation*}
    which implies, 
    \begin{equation*}
        \sum_{i=0}^n (i+1)x_i \le (n+1)2^{-a_n}.
    \end{equation*}
     This is clearly a contradiction, as $n > 1$ and $\seq{x_n}$ is positive. Thus, if $k = a_n$ then $n \le \phi(2^{-k})$. So we can determine whether $k \in A$, by computably searching the first $\phi(2^{-k})$ terms in $\seq{a_n}$. 
\end{example}
The above construction demonstrates that a general rate of convergence for Kronecker's lemma must depend on a rate of convergence for $\sum_{i=0}^\infty x_i$. In fact, more is true.

Following \cite{simpson2009subsystems}, let $\mathrm{RCA_0}$ be the standard base system of Reverse 
mathematics (the subsystem of second order arithmetic containing only $\Sigma_1^0$ induction and $\Delta_1^0$ comprehension) and the system $\mathrm{ACA_0}$ which extends $\mathrm{RCA_0}$ by arithmetic comprehension \footnote{Arithmetic comprehension is the scheme \[
\exists X\, \forall n\, (n \in X \leftrightarrow \phi(n))
\]
for all arithmetic formulas $\phi$ (formulas with no bound set variables) without $X$ occurring as a free variable.}.

Here, we shall be working with the language of second-order arithmetic, where we quantify over variables representing natural numbers and subsets of natural numbers. Furthermore, via the numerically defined paring function 
\[
(m,n):= (m+n)^2 + m
\]
we can encode the integers, rationals and reals. In addition, as standard, for set variables  $X, Y, f \in 2^{\NN}$, $f: X \to Y$ is shorthand for 
\[
\forall l, n, m \in \NN \,((l,n) \in f \land (l,m) \in f \to n = m)
\]
and for a formula $\phi(f)$ in the language of second order arithmetic, 
\[
\forall f:X \to Y \,( \phi(f))
\]
is shorthand for 
\[
\forall f \in 2^\NN \,((f:X \to Y) \to \phi(f))
\]
with a similar convention for 
\[
\exists f:X \to Y \, (\phi(f)).
\]
Moreover, we  write $f(n)=m$ for $(n,m) \in f$.

In what follows, we need the standard fact that $\mathrm{ACA}_0$ proves the axiom of choice on arithmetic formulas. More specifically.

\begin{proposition}
    The axiom of choice for arithmetic formulas is provable in $\mathrm{ACA}_0$. That is, the following holds:
\[
\forall n \in \NN \,\exists m \in \NN \,\phi(n,m) \to \exists f : \NN \to \NN \, \forall n \, \phi(n,f(n))
\]
for all arithmetic formulas $\phi$, without $f$ occurring free. Here $\phi(n,f(n))$ is shorthand for $\exists m\, (f(n) = m \land \phi(n,m)))$.
\end{proposition}
\begin{proof}
    By arithmetic comprehension \footnote{A direct application of  Arithmetic comprehension yields 
 \[
\exists f \in 2^{\NN}\, \forall n \in \NN \,(n \in f  \leftrightarrow (\exists i,j \in \NN\,(n = (i,j) \land (\phi(i,j) \land( \forall j_0 \in \NN \,(\phi(i,j_0) \to j \le j_0))))))
\]
 and our stated formula follows.}, we have 
\[
\exists f \in 2^{\NN}\, \forall n,m \in \NN \,((n,m) \in f \leftrightarrow (\phi(n,m) \land( \forall m_0 \in \NN \,(\phi(n,m_0) \to m \le m_0)))).
\]
Taking such an $f$, one can easily show $f: \NN \to \NN$ and satisfies the consequence of the implication in the statement of the result.
\end{proof}
Now, for sequences of real numbers $\seq{x_n}, \seq{a_n}$ let 
\begin{equation*}
\begin{aligned}
\mathrm{KRON}(\seq{a_n},\seq{x_n})&:\equiv \forall n\in \NN\,(0<a_n \le a_{n+1}) \land \forall m\in \NN\, \exists k\in \NN\, (a_k \ge m)\land \left(\left\{\sum_{i=0}^nx_n\right\}\mbox{ is Cauchy}\right)\\
&\to\frac{1}{a_n}\sum_{i=0}^n a_ix_i \mbox{ converges to $0$} 
\end{aligned}
\end{equation*}
and
\begin{equation*}
\begin{aligned}
\mathrm{RKRON}(\seq{a_n},\seq{x_n})&:\equiv \forall n\in \NN\,(0<a_n \le a_{n+1}) \land \forall m\in \NN\, \exists k\in \NN\, (a_k \ge m)\land \left(\left\{\sum_{i=0}^nx_n\right\}\mbox{ is Cauchy}\right)\\
&\to \exists g:\NN\to \NN \, \left(\mbox{$g$ is a rate of convergence to $0$ for }\frac{1}{a_n}\sum_{i=0}^n a_ix_i \right).
\end{aligned}
\end{equation*}
Here we use the ` $2^{-k}$' formulation of convergence, for example, by ` $g$ is a rate of convergence to $0$ for $\frac{1}{a_n}\sum_{i=0}^n a_ix_i$' we mean 
\begin{equation*}
    \forall k \in \NN\, \forall n \in \NN \left(n \ge g(k) \to \left| \frac{1}{a_n}\sum_{i=0}^n a_ix_i\right|\le 2^{-k} \right).
\end{equation*}
We have the following:
\begin{theorem}
\label{thrm:Kronreverse}
In $\mathrm{RCA_0}$ we have the following:
    \begin{itemize}
        \item[(i)] For all sequences of reals $\seq{x_n}$ and $\seq{a_n}$, $\mathrm{KRON}(\seq{a_n},\seq{x_n})$.
        \item [(ii)] $\mathrm{ACA_0}$ implies, for all sequences of reals $\seq{x_n}$ and $\seq{a_n}$, $\mathrm{RKRON}(\seq{a_n},\seq{x_n})$.
        \item[(iii)] For all sequences of positive rationals $\seq{x_n}$ $(\mathrm{RKRON}(\seq{n+1},\seq{x_n}))$ implies $\mathrm{ACA_0}$.  
    \end{itemize}
\end{theorem}
\begin{proof}
    Write $s_n:= \sum_{i=0}^nx_i$.
    
    For (i), suppose we have a sequence of reals $\seq{a_n}$ that is increasing, positive, and satisfies $\forall m\in \NN\, \exists k\in \NN\, (a_k \ge m)$, as well as a sequence of reals $\seq{x_n}$, such that $\seq{s_n}$ is Cauchy. Now $\mathrm{RCA_0}$ proves that Cauchy sequences are bounded so that we can take $S \in \NN$ such that $\forall n \in \NN \, (|s_n|<S)$. Let $k \in \NN$ be given. By the Cauchy property of $\seq{s_n}$, we may take $n \in \NN$ such that, for all $m \ge n$ we have $|s_m - s_n|\le 2^{-k-2}$. Now taking $W$ such that $a_W \ge 2^{k+2}a_nS$ and following the proof of Theorem \ref{thrm:banachkron} implies that for $m \ge \max\{W,n\}$,  $\left| \frac{1}{a_m}\sum_{i=0}^m a_ix_i\right|\le 2^{-k}$.

    For (ii), we have that for an increasing, positive sequence of reals $\seq{a_n}$ satisfying $\forall m\in \NN\, \exists k\in \NN\, (a_k \ge m)$, and a sequence $\seq{x_n}$ with $\seq{s_n}$ Cauchy, part (i) implies that in $\mathrm{RCA_0}$ we can prove $\frac{1}{a_n}\sum_{i=0}^n a_ix_i$ converges to $0$, and the result follows from an application of the axiom of choice on arithmetic formulas. 
    
    For (iii), we demonstrate that we can construct the range of a given one-to-one function $f:\NN\to \NN$, and the result follows from \cite[Lemma III.1.3]{simpson2009subsystems}. Take such an $f:\NN\to \NN$ and let $x_i:= 2^{-f(i)}$. Then, since $RCA_0$ proves that monotone bounded sequences are Cauchy, we have $\seq{s_n}$ is Cauchy. Therefore, by $(\mathrm{RKRON}(\seq{n+1},\seq{x_n}))$, we have $g:\NN\to \NN$ satisfying 
\begin{equation*}
    \forall k \in \NN\, \forall n \in \NN \left(n \ge g(k) \to  \frac{1}{n+1}\sum_{i=0}^n (i+1)2^{-f(i)}\le 2^{-k} \right).
\end{equation*}
One then has that for all $k \in \NN$,
\begin{equation}
\label{eqn:impACA}
    (\exists n\,(f(n) = k))\leftrightarrow (\exists n \le g(k)\,(f(n)= k)),
\end{equation}
the backwards implication is clear. For the forward implication, if we do not have $(\exists n \le g(k)\,(f(n)= k))$ but $(\exists n\,(f(n) = k))$, then $(\exists n > g(k)\,(f(n)= k))$, and for such an $n$, we have 
\begin{equation*}
\frac{1}{n+1}\sum_{i=0}^n (i+1)2^{-f(i)}\le 2^{-k}=2^{-f(n)} 
\end{equation*}
a contradiction. Thus, (\ref{eqn:impACA}) and   $\Delta_1^0$ comprehension implies $\exists X \in 2^\NN\, \forall k \in \NN\, (k \in X \leftrightarrow (\exists n\,(f(n) = k)))$ and the result follows.
\end{proof}
\begin{remark}
\label{rem:rom_mono}
    Example \ref{thrm:kroncomp} demonstrates that when we assume that $\seq{x_n}$is nonnegative, a general computable rate of convergence for Kronecker's lemma, which depends on a computable bound for $\sum_{i=0}^\infty x_i$, cannot exist.

    However, it is a Folklore result in applied proof theory that if $\seq{a_n}$ be a nondecreasing sequence of nonnegative numbers bounded above by $L >0$. Then 
    $$\Phi(\varepsilon,g) = g^{\floor{\frac{L}{\varepsilon}}}(0)$$
    is a rate of metastable convergence for $\seq{a_n}$ (see \cite{kohlenbach:08:book}).

    When $\seq{x_n}$ is a nonnegative sequence, the partial sums will form a nondecreasing sequence. So, although we cannot hope to find a rate of convergence for $\frac{1}{a_n}\sum_{i=0}^n a_ix_i \to 0$ that just depends on a bound for $\sum_{i=0}^\infty x_i$, such a bound would give us a rate of metastability for the convergence of the partial sums of $\seq{\sum_{i=0}^nx_i}$ and we can use Corollary \ref{cor:quantkron} to obtain a rate of metastability for $\frac{1}{a_n}\sum_{i=0}^n a_ix_i \to 0$.
\end{remark}
\section{The computational content of Chung's Law of Large Numbers on Banach spaces}
\label{sec:chung}
The following generalisation of Theorem \ref{thrm:chung} is due to Woyczynski \cite{woyczynski1974random}:

\begin{theorem}
\label{thrm:chung:banach}
   Suppose $\seq{X_n}$ is a sequence of independent random variables taking values in a type $p$ Banach space $\BB$  with $\EE(X_n) = 0$ for all $n \in \NN$. Let $0< a_0 \le a_1 \le \cdots$ be such that $a_n \to \infty$. For each $n \in \NN$, let $\phi_n : \RR^+ \to \RR^+$ be a function such that
   \begin{equation}
      \frac{\phi_n(t)}{t} \mbox{ and } \frac{t^p}{\phi_n(t)}
   \end{equation}
 are nondecreasing and assume
 $$ \sum_{n=1}^\infty \frac{\EE(\phi_n(|X_n|))}{\phi_n(n)} < \infty.$$
Then $\frac{S_n}{n} \to 0$ almost surely.
\end{theorem}
In this section, we shall give a computational interpretation to the above theorem by constructing rates of metastability in the conclusion in terms of suitable computational interpretations given to the assumptions in the premise of the theorem. To do this, we must first obtain a computational interpretation of the probabilistic analogue of Kronecker's lemma, which is a core component in the proof of the aforementioned result, as well as many other Strong Law of Large Numbers results. 
\subsection{Qunatitative probabilistic Kronecker's lemma}
\label{subsec:qauntprobkron}
We shall give a computational interpretation to the analogue of Theorem \ref{thrm:realprobkron} for random variables taking values in $\BB$. Recall from the Section \ref{subsec:probonbanach} that we now take $\BB$ to be a separable Banach space. The proof of Theorem \ref{thrm:realprobkron} was very straightforward, however obtaining a computational interpretation for this result will require a bit of work. This is because the computational interpretation given to almost sure convergence is actually a computational interpretation of almost uniform convergence. Thus, to obtain a quantitative version of the probabilistic Kronecker's lemma, one must analyse the following result:
\begin{theorem}[Probabilistic Kronecker's lemma on $\BB$]
    \label{thrm:probkron}
  Let $\seq{Y_n}$ be a sequence of random variables taking values in $\BB$ and $0< a_0 \le a_1 \le \cdots$ be such that $a_n \to \infty$. 
If $\sum_{i=0}^\infty Y_i < \infty$ almost uniformly , then 
    \begin{equation*}
        \frac{1}{a_n}\sum_{i=0}^n a_iY_i \to 0
    \end{equation*}
    almost uniformly.  
\end{theorem}
\begin{proof}
     \begin{equation*}
     \begin{aligned}
         & \sum_{i=0}^\infty Y_n < \infty \textrm{ $\mathrm{a.u}$} \to \sum_{=0}^\infty Y_n < \infty \textrm{ $\mathrm{a.s}$} \\
         &\to \frac{1}{a_n}\sum_{i=0}^n a_iY_i \to 0 \textrm{ $\mathrm{a.s}$} \to \frac{1}{a_n}\sum_{i=0}^n a_iY_i \to 0 \textrm{ $\mathrm{a.u}$}
     \end{aligned}
 \end{equation*}
 The first and last implication follows from the fact that almost sure uniform convergence is equivalent to almost sure convergence. This result is sometimes known as Ergorov's theorem and can be obtained through repeated applications of \cite[Lemma 3.1]{NeriPowell:martingales:2024}. The second implication follows by a pointwise application of Kronecker's lemma.
\end{proof}
Analysing the above proof to obtain a computational interpretation of the probabilistic Kronecker's lemma will involve a quantitative analysis of the step evoking Egorov's theorem. There does exist such a quantitative theorem \cite{Avigad-Dean-Rute:Dominated:12}; however, this would correspond to a solution of bar recursive complexity. 

Due to the uniformities of the rates obtained from Kronecker's lemma, one can get a quantitative probabilistic Kronecker's lemma that avoids the need for a quantitative analysis of Egorov's theorem.

\begin{theorem}[Finitary probabilistic Kronecker's lemma]
\label{thrm:quantprobkron1}
Let $\seq{Y_n}$ be a sequence of $\BB$ valued random variables, set $Z_n := \sum_{i=0}^n Y_i$. Let $0< a_0 \le a_1 \le \cdots$ be such that $a_n \to \infty$ and $f_{\seq{a_n}}$ be as in Theorem \ref{thrm:kroncomp}. 

    For every $\psi:\QQ^+ \times \QQ^+ \to \NN$, sequence of natural numbers $\seq{z_n}$, $\varepsilon, \lambda \in \QQ^+$ and $k \in \NN$, if  $M := \psi(\frac{\varepsilon}{4},\frac{\lambda}{2}) = \gamma_{\frac{\lambda}{2}}(\frac{\varepsilon}{4})$ satisfies,  \begin{equation*}
         \mathbb{P}\left(\bigcup_{i=0}^M\left\{\norm{Z_i} \ge z_M\right\}\right) \le \frac{\lambda}{2}
    \end{equation*}
    and
    \begin{equation*}
       \mathbb{P}\left( \max_{M \le m \le k}\norm{Z_M -  Z_m} \ge \frac{\varepsilon}{4}\right) < \frac{\lambda}{2}
    \end{equation*} 
  then $N := \Psi_{\seq{a_n}, \seq{z_n}}(\psi,\varepsilon,\lambda)$ satisfies,
    \begin{equation*}
         \mathbb{P}\left( \max_{N \le n \le k}\bignorm{\frac{1}{a_n}\sum_{i=0}^n a_iY_i} \ge \varepsilon\right) < \lambda.
    \end{equation*}
    Where,
$$\Psi_{\seq{a_n}, \seq{z_n}}(\psi,\varepsilon,\lambda) := \max\left\{Q,f_{\seq{a_n}}\left(\frac{4a_Qz_Q}{\varepsilon}\right)\right\}= \Gamma_{\seq{a_n}}\left(\gamma_{\frac{\lambda}{2}},\seq{z_n},\varepsilon\right)$$
($\Gamma$ as defined in Theorem \ref{thrm:quantkron}) and,
    \begin{equation*}
        \begin{aligned}
            &Q := \psi\left(\frac{\varepsilon}{4},\frac{\lambda}{2}\right)\\
            & \gamma_\lambda(\varepsilon):= \psi(\varepsilon,\lambda).
        \end{aligned}
    \end{equation*}
\end{theorem}
\begin{proof}
Let  $\psi,\seq{z_n}, \varepsilon, \lambda, k$ satisfying the premise of the theorem be given. We have,
\begin{equation*}
        \begin{aligned}
            &\mathbb{P}\left(\max_{N \le n \le k}\bignorm{\frac{1}{a_n}\sum_{i=0}^n a_iY_i} \ge \varepsilon\right) \\
            &= \mathbb{P}\left(\left\{\max_{N \le n \le k}\bignorm{\frac{1}{a_n}\sum_{i=0}^n a_iY_i} \ge \varepsilon\right\} \cap \bigcup_{i=1}^M\left\{\norm{Z_i} \ge z_M\right\}\right) 
            \\ 
            &+\mathbb{P}\left(\left\{\max_{N \le n \le k}\bignorm{\frac{1}{a_n}\sum_{i=0}^n a_iY_i} \ge \varepsilon\right\} \cap \bigcap_{i=1}^M\left\{\norm{Z_i} < z_M\right\}\right)  \\
            &\le \frac{\lambda}{2} + \mathbb{P}\left(\left\{\max_{N \le n \le k}\bignorm{\frac{1}{a_n}\sum_{i=0}^n a_iY_i} \ge \varepsilon\right\} \cap \bigcap_{i=1}^M\left\{\norm{Z_i} < z_M\right\}\right). 
        \end{aligned}
    \end{equation*}
    Suppose 
$$\max_{N \le n \le k}\bignorm{\frac{1}{a_n}\sum_{i=0}^n a_iY_i} \ge \varepsilon \land \bigwedge_{i=1}^M\left(\norm{Z_i} < z_M\right),$$
     but for all $m \in [M, k], |Z_m -  Z_M| < \frac{\varepsilon}{4}$. Theorem \ref{thrm:quantkron} implies that $$\forall n \in [N,k]\bignorm{\frac{1}{a_n}\sum_{i=0}^n a_iY_i} < \varepsilon, $$ 
     (recalling $N =  \Gamma_{\seq{a_n}}(\gamma_{\frac{\lambda}{2}},\seq{z_n},\varepsilon)$), which is a contradiction. So, 
     $$\max_{M \le m \le k}\norm{Z_M -  Z_m} \ge \frac{\varepsilon}{4}.$$ 
     This implies,
    \begin{equation*}
        \begin{aligned}
         &\mathbb{P}\left(\left\{\max_{N \le n \le k}\bignorm{\frac{1}{a_n}\sum_{i=0}^n a_iY_i} \ge \varepsilon\right\} \cap \bigcap_{i=1}^M\left\{\norm{Z_i} < z_M\right\}\right) \\ &\qquad\qquad\le\mathbb{P}\left( \max_{M \le m \le k}\norm{Z_M -  Z_m} \ge \frac{\varepsilon}{4}\right) < \frac{\lambda}{2}
        \end{aligned}
    \end{equation*}
    So we are done.
 \end{proof}
Theorem \ref{thrm:quantprobkron1} allows us to immediately obtain a quantitative version of the probabilistic Kronecker's lemma, where we obtain rates in the conclusion given rates in the premise.
\begin{corollary}
\label{cor:qauntprobkron}
Let $\seq{a_n}$, $\seq{Z_n}$ be as in Theorem \ref{thrm:quantprobkron1} and for each $\lambda \in \QQ^+$, let $\seq{z_n(\lambda)}$ be a sequence of nondecreasing natural numbers satisfying, for all $n \in \NN$,
    \begin{equation}
    \label{eqn:condition}
    \mathbb{P}\left(\bigcup_{i=0}^n\left\{\norm{Z_i} \ge z_n(\lambda)\right\}\right)\le \lambda
    \end{equation}
    for all $n \in \NN$. 
    
    Now suppose, $\sum_{i = 0}^n Y_i$
    converges almost surely with a rate of metastable almost sure convergence $\Phi$. Then 
    $$ \frac{1}{a_n}\sum_{i=0}^n a_iY_i$$ converges to $0$ almost surely, with rate of metastable almost sure convergence
    \begin{equation*}
        \kappa^p_{\Phi,\seq{a_n},\seq{z_n}}(\varepsilon,\lambda,K) = \max\left\{Q,f_{\seq{a_n}}\left(\frac{4a_Qz_Q(\lambda/2)}{\varepsilon}\right)\right\},
    \end{equation*}
   where, 
   \begin{equation*}
        Q := \Phi\left(\frac{\varepsilon}{4},\frac{\lambda}{2},H\right)
\end{equation*}
and
    \begin{equation*}
         H := H_{\varepsilon,\lambda,K,\seq{a_n},\seq{z_n}}(n) :=  \tilde{K}\left(\max\left\{n,f_{\seq{a_n}}\left(\frac{4a_nz_n(\lambda/2)}{\varepsilon}\right)\right\}\right), 
    \end{equation*}
    with $\tilde{K}(n)=n+K(n)$.
\end{corollary}
\begin{proof}
Let  $\varepsilon,\lambda \in \QQ^+$ and $K:\NN \to \NN$ be given. There exists $M \le Q:= \Phi(\frac{\varepsilon}{4},\frac{\lambda}{2},H)$ such that 
  \begin{equation*}
        \mathbb{P}\left( \max_{M \le m \le H(M)}\norm{Z_M -  Z_m} \ge \frac{\varepsilon}{4}\right)\le \mathbb{P}\left( \max_{M \le m \le M+ H(M)}\norm{Z_M -  Z_m} \ge \frac{\varepsilon}{4}\right) < \frac{\lambda}{2}
    \end{equation*}
 by the definition of a rate of metastability. Let $\psi(\delta_1,\delta_2) = M$, for all $\delta_1,\delta_2 \in \QQ^+$. Then 
 $$N = \Psi_{\seq{a_n}, \seq{z_n(\lambda/2)}}(\psi,\varepsilon,\lambda) = \max\left\{M,f_{\seq{a_n}}\left(\frac{4a_Mz_M(\lambda/2)}{\varepsilon}\right)\right\}\le \max\left\{Q,f_{\seq{a_n}}\left(\frac{4a_Qz_Q(\lambda/2)}{\varepsilon}\right)\right\}$$ 
 and $k =H(M) = N+K(N)$  satisfies, 

   \begin{equation*}
         \mathbb{P}\left( \max_{N \le n \le N+K(N)}\norm{\frac{1}{a_n}\sum_{i=0}^n a_iY_i} \ge \varepsilon\right) < \lambda
    \end{equation*}
    so we are done.
\end{proof}

\begin{remark}
\label{rem:probroc}
Note that as in the deterministic case (see Remark \ref{rem:rocimpram}), if $\Phi$ above is a rate of convergence, then we get a rate of almost sure convergence to zero given by the above expression, but with 
\begin{equation*}
           Q := \Phi\left(\frac{\varepsilon}{4},\frac{\lambda}{2}\right).
\end{equation*}
\end{remark}

Condition ($\ref{eqn:condition}$) can naturally be satisfied if we give a suitable computational interpretation to the finiteness of random variables.
\begin{proposition}
    A real-valued random variable $Y$ is finite almost everywhere iff
$$\PP(|Y| \ge m) \to 0$$
as $m \to \infty.$
\end{proposition}
\begin{proof}
   The events 
    $$A_m = \{\omega \in \Omega: |Y(\omega)| \ge m\}$$
    form a decreasing sequence of events; thus, by the continuity of the probability measure, we have, 
    $$\PP(A_m) \to \PP\left(\bigcap_{m = 0}^\infty A_m\right)= \PP(\{\omega: |Y(\omega)| = \infty\})$$
as $m \to \infty$, and the result follows.
\end{proof}
The above result can also be obtained through an application of \cite[Lemma 3.1]{NeriPowell:martingales:2024} by noting that the random variable $Y$ being almost surely finite on the element $\omega \in \Omega$ is equivalent to the formula $\exists N\, (|Y(\omega)| \le N)$ which satisfies the monotonicity requirement of \cite{NeriPowell:martingales:2024}. We now have the following definition:
\begin{definition}
\label{def:asf}
$R: \QQ^+ \to \NN$ is a rate of almost sure finiteness if it is a rate of convergence for 
$$\PP(|Y| \ge m) \to 0,$$ 
that is, if it satisfies,
$$\forall \varepsilon \in \QQ^+ \,(\PP(|Y| \ge R(\varepsilon)) \le \varepsilon)$$
\end{definition}
\begin{example}
\label{ex:fin_rate}
    If a random variable, $Y$, is integrable then we have 
$$\PP(|Y| \ge m) \le \frac{\EE(|Y|)}{m}$$
for all $m >0$, by Markov's inequality. Thus, 
$$R(k):= \frac{E}{\varepsilon}$$
is a rate of almost sure finiteness, for all $E \in \NN$ satisfying $E \ge \EE(|Y|)$.
\end{example}
\begin{remark}
\label{rem:finZ}
    Now we have if $\seq{Z_n}$ have respective rates of almost finiteness $\seq{R_n}$, then for any $k\in \NN$ we can take 
$$z_n(\lambda):= \max_{i \le n} R_i\left(\frac{\lambda}{n}\right)$$
in Corollary \ref{cor:qauntprobkron}, as 
\begin{equation*}
    \mathbb{P}\left(\bigcup_{i=0}^n\left\{|Z_i| \ge z_n(\lambda)\right\}\right)\le \sum_{i=0}^n\mathbb{P}\left(\left\{|Z_i| \ge z_n(\lambda)\right\}\right)\le \sum_{i=0}^n\mathbb{P}\left(\left\{|Z_i| \ge R_i(\lambda/n)\right\}\right) \le \lambda.
\end{equation*}
Furthermore, if there exists a function $R:\QQ^+ \to \NN$ such that 
\begin{equation*}
\mathbb{P}\left(\bigcup_{i=0}^\infty\left\{|Z_i| \ge R(\lambda)\right\}\right) \le \lambda,
\end{equation*}
we can take $z_n(\lambda):= R(\lambda)$. Such an $R$ is known as a modulus of uniform boundedness for $\seq{Z_n}$ and was introduced in \cite{NeriPowell:martingales:2024}.
\end{remark}

\subsection{Rates for Chung's Law of Large Numbers on Banach spaces}
\label{subsec:chungrates}
Throughout this section, assume $\BB$ is a type $p$ Banach space with a constant $C$ satisfying (\ref{eqn:Woy}). 
In this section, we shall give a quantitative version of Chung's Strong Law of Large Numbers on type $p$ Banach spaces due to Woyczynski. The result follows from some lemmas, the first of which is a generalisation of Kolmogorov's inequality.
\begin{lemma} [Kolmogorov’s inequality, cf. Theorem 3.2.4B of \cite{delporte1964functions}]
\label{thrm:Kolineq}
Let $\seq{X_n}$ be a sequence of independent random variables taking values in $\BB$, each with expected value $0$. Setting $S_n := \sum_{i=0}^n X_i$, we have, for all $n \in \NN$, $\varepsilon > 0$ and $r \ge 1$,
\begin{equation*}
    \mathbb{P}\left(\max_{1 \le i \le n}\norm{S_i} > \varepsilon\right) \le \frac{\EE(\norm{S_n}^r)}{\varepsilon^r}. 
\end{equation*}
\end{lemma}
We omit the proof of this result. We now need quantitative versions of \cite[Theorems 2 and 2a]{woyczynski1974random}.
\begin{theorem}[Quantitative version of Theorem 2 of \cite{woyczynski1974random}]
\label{thrm:quant2}
Let $\seq{X_n}$ be a sequence of independent random variables taking values in $\BB$, each with expected value $0$. Suppose $\sum_{i=0}^n\EE(\phi_0(\norm{X_i}))$ converges with rate of Cauchy metastability $\Phi$, where $\phi_0(t) = t^p$ for $0 \le t \le 1$ and $\phi_0(t) = t$ for $t > 1$. Then $S_n$ converges almost surely with rate of metastable almost sure convergence 
\begin{equation*}
    \Delta_{\Phi}(\varepsilon,\lambda,K) = \Phi(\Tilde{\varepsilon},K),
\end{equation*}
where 
$$\Tilde{\varepsilon} := \min\left\{\frac{\varepsilon\lambda}{6}, \frac{\lambda\varepsilon^p}{2^{3p-1}3C},\left(\frac{\lambda\varepsilon^p}{2^{2p-1}3}\right)^{\frac{1}{p}} \right\}$$

and $S_n := \sum_{i=0}^n X_i$.
\end{theorem}
\begin{proof}
As in the proof presented in \cite{woyczynski1974random}, for each $i\in \NN$, let $X_i'= X_i1_{\{\norm{X_i}\le 1\}}$ and $X_i''= X_i1_{\{\norm{X_i}> 1\}}$. Clearly $X_i = X_i'+X_i''$ and $\seq{X_n'}$ and $\seq{X_n''}$ are independent random variables taking values in $\BB$. Let $S_n' = \sum_{i=0}^nX_i'$ and $S_n'' = \sum_{i=0}^nX_i''$.
   Suppose $\varepsilon > 0, \lambda > 0, K:\mathbb{N} \to \mathbb{N}$ are given, we have $N \le \Phi(\Tilde{\varepsilon},K)$ such that
   $$\sum_{i=N+1}^{K(N)+N}\EE(\phi_0(\norm{X_i})) < \Tilde{\varepsilon}.$$
   
   Now setting $K:= K(N)+N$ gives,
 \begin{equation*}
  \begin{aligned}
     \mathbb{P}\left(\max_{N \le n \le K}\norm{S_n -  S_N} > \varepsilon\right)
     &\le  \mathbb{P}\left(\max_{N \le n \le K}\norm{S_n' -  S_N' + S_n'' -  S_N''} > \varepsilon\right)\\
    &\le  \mathbb{P}\left(\max_{N \le n \le K}\left(\norm{S_n' -  S_N'} + \norm{S_n'' -  S_N''}\right) > \varepsilon\right)\\   &\le\mathbb{P}\left(\max_{N \le n \le K}\norm{S_n' -  S_N'}  > \frac{\varepsilon}{2} \lor \max_{N \le n \le K}\norm{S_n'' -  S_N''}  > \frac{\varepsilon}{2}\right)\\
   &\le\mathbb{P}\left(\max_{N \le n \le K}\norm{S_n' -  S_N'}  > \frac{\varepsilon}{2}\right)+\mathbb{P}\left( \max_{N \le n \le K}\norm{S_n'' -  S_N''}  > \frac{\varepsilon}{2}\right).
  \end{aligned}
 \end{equation*}
Now by Theorem  \ref{thrm:Kolineq} and the definition of $\phi_0$, we have
\begin{equation*}
\begin{aligned}
&\mathbb{P}\left(\max_{N \le n \le K}\norm{S_n'' -  S_N''}  > \frac{\varepsilon}{2}\right) \le \frac{2\EE\left(\norm{\sum_{i=N+1}^{K}X_i''}\right)}{\varepsilon}\\
&\qquad\qquad\frac{2\EE\left(\sum_{i=N+1}^{K}\norm{X_i''}\right)}{\varepsilon} \le \frac{2\sum_{i=N+1}^{K}\EE(\phi_0(\norm{X_i}))}{\varepsilon}
\end{aligned}
\end{equation*}
and we have
\begin{equation*}
    \begin{aligned}
       & \mathbb{P}\left(\max_{N \le n \le K}\norm{S_n' -  S_N'}  > \frac{\varepsilon}{2}\right)
        \le \frac{2^p}{\varepsilon^p}\EE\left(\norm{\sum_{i=N+1}^{K}X_i'}^p\right)\\
        &\qquad\qquad\le\frac{2^p}{\varepsilon^p}\EE\left(\left(\norm{\sum_{i=N+1}^{K}(X_i'-\EE(X_i'))} +  \norm{\sum_{i=N+1}^{K}\EE(X_i')} \right)^p\right)\\
        &\qquad\qquad\le\frac{2^{2p-1}}{\varepsilon^p}\EE\left(\norm{\sum_{i=N+1}^{K}(X_i'-\EE(X_i'))}^p + \norm{\sum_{i=N+1}^{K}\EE(X_i')}^p\right)\\
        &\qquad\qquad\le\frac{2^{2p-1}}{\varepsilon^p}\left(C\sum_{i=N+1}^{K}\EE(\norm{X_i'-\EE(X_i')}^p) + \norm{\sum_{i=N+1}^{K}\EE(X_i'')}^p\right)\\
        &\qquad\qquad\le\frac{2^{2p-1}}{\varepsilon^p}\left(C2^{p-1}\left(\sum_{i=N+1}^{K}\EE(\norm{X_i'}^p) + \sum_{i=N+1}^{K(N)}\EE(\norm{X_i'}^p)\right) + \left(\EE\left(\norm{\sum_{i=N+1}^{K}X_i''}\right)\right)^p\right)\\
    \end{aligned}
\end{equation*}
\begin{equation*}
    \begin{aligned}
        &\le\frac{2^{2p-1}}{\varepsilon^p}\left(C2^{p}\left(\sum_{i=N+1}^{K}\EE(\phi_0(\norm{X_i}))\right) + \left(\sum_{i=N+1}^{K}\EE(\phi_0(\norm{X_i}))\right)^p\right)\\
        &\qquad\qquad = \frac{2^{3p-1}C}{\varepsilon^p}\sum_{i=N+1}^{K}\EE(\phi_0(\norm{X_i})) + \frac{2^{2p-1}}{\varepsilon^p}\left(\sum_{i=N+1}^{K}\EE(\phi_0(\norm{X_i}))\right)^p.
    \end{aligned}
\end{equation*}
The first inequality follows from Theorem \ref{thrm:Kolineq}, the second inequality by the triangle inequality, and the third inequality follows from the fact that if $a,b>0$, then
\begin{equation}
    \label{eq:jensen}
    \left(\frac{a+b}{2}\right)^p \le \frac{a^p + b^p}{2}.
\end{equation}
 The fourth inequality follows from the fact that $\BB$ is of type $p$ and $0 = \EE(X_i) = \EE(X_i') + \EE(X_i'')$, for all $i \in \NN$. The fifth inequality follows from the triangle inequality and (\ref{eq:jensen}). The sixth inequality follows from the definition of $\phi_0$.\\

 Putting all of this together and using the definition of $\varepsilon$, we get 
 \begin{equation*}
 \begin{aligned}
     & \mathbb{P}\left(\max_{N \le n \le K}\norm{S_n -  S_N} > \varepsilon\right) \le \frac{2\sum_{i=N+1}^{K}\EE(\phi_0(\norm{X_i}))}{\varepsilon}  + \\
     &\qquad\qquad\frac{2^{3p-1}C}{\varepsilon^p}\sum_{i=N+1}^{K}\EE(\phi_0(\norm{X_i})) + \frac{2^{2p-1}}{\varepsilon^p}\left(\sum_{i=N+1}^{K}\EE(\phi_0(\norm{X_i}))\right)^p \le \lambda.
 \end{aligned}
 \end{equation*}
\end{proof}
 From the above, we immediately get 
 \begin{corollary}[Quantitative version of Theorem 2a of \cite{woyczynski1974random}]
 \label{thrm:quant2a}
 Let $\seq{X_n}$ be a sequence of independent random variables taking values in $\BB$, each with expected value $0$. Let $0< a_0 \le a_1 \le \cdots$ be such that $a_n \to \infty$. Further suppose we have a sequence of functions $\seq{\phi_n: \RR^+ \to \RR^+}$ such that,
     \begin{equation*}
         \frac{\phi_n(t)}{t} \mbox{ and } \frac{t^p}{\phi_n(t)} \mbox{ are nondecreasing for all } n \in \NN.
     \end{equation*}
     If we have,
     $$\sum_{k=0}^\infty \frac{\EE(\phi_k(\norm{X_k}))}{\phi_k(a_k)} < \infty$$ 
     and converges with rate of Cauchy metastability $\Phi$,
     
     then 
     $$\sum_{k=0}^n\frac{X_k}{a_k}$$ converges with rate of Cauchy metastability $\Delta_{\Phi}$, where $\Delta_\Phi$ is defined as in Theorem \ref{thrm:quant2}. 
 \end{corollary}
 \begin{proof}
  For each $n \in \NN$, let
  $$\Gamma_n(t) =\frac{\phi_n(a_nt)}{\phi_n(a_n)}.$$
  It is easy to see that for every function $\Gamma$ with $\Gamma(1) = 1$ and both
  \begin{equation*}
      \frac{\Gamma(t)}{t}\mbox{ and } \frac{t^p}{\Gamma(t)}
  \end{equation*}
  nondecreasing, we have $\Gamma(t) \ge \phi_0(t)$ for all $t \ge 0$. One can easily check that for each $n \in \NN$, $\Gamma_n$ satisfies this property. Thus, a rate of convergence for 
  $$\sum_{k=0}^n \frac{\EE(\phi_k(\norm{X_k}))}{\phi_k(a_k)}= \sum_{k=0}^n\EE\left(\Gamma_k\left(\norm{\frac{X_k}{a_k}}\right)\right)=$$ 
  will be a rate of convergence for $$\sum_{k=0}^n\EE\left(\phi_0\left(\norm{\frac{X_k}{a_k}}\right)\right)$$
  and the result follows from Theorem \ref{thrm:quant2}.
 \end{proof}
 We can now prove a quantitative version of the main result of \cite{woyczynski1974random}.
 \begin{theorem}
 \label{thrm:quantchung}
  Let $\seq{X_n}$ be a sequence of independent random variables taking values in $\BB$, each with expected value $0$.  Let $0< a_0 \le a_1 \le \cdots$ be such that $a_n \to \infty$ and  
$f_{\seq{a_n}}(x) := \min\{n \in \NN : a_n \ge x\}$. Further suppose we have a sequence of functions $\seq{\phi_n: \RR^+ \to \RR^+}$ such that,
     \begin{equation}
     \label{eqn:chungp}
         \frac{\phi_n(t)}{t} \mbox{ and } \frac{t^p}{\phi_n(t)} \mbox{ are nondecreasing for all } n \in \NN.
     \end{equation}
 For each $\lambda \in \QQ^+$, let $\seq{z_n(\lambda)}$ be a sequence of natural numbers satisfying,
\begin{equation*}
\mathbb{P}\left(\bigcup_{i=0}^n\left\{\norm{\sum_{k= 0}^i\frac{X_k}{a_k}} \ge z_n(\lambda)\right\}\right)\le \lambda
    \end{equation*}
    for all $n \in \NN$.  
 Suppose 
 $$\sum_{k=0}^\infty \frac{\EE(\phi_k(\norm{X_k}))}{\phi_k(a_k)} < \infty$$
 and converges with rate of Cauchy metastability $\Phi$. Then $\frac{S_n}{n}$ converges to 0 almost surely with a rate of metastable almost sure convergence
 $$\kappa^P_{\Delta_{\Phi}, \seq{a_n},\seq{z_n}}$$ 
 \end{theorem}
 \begin{proof}
     $$\sum_{k=0}^n \frac{\EE(\phi_k(\norm{X_k}))}{\phi_k(a_k)}$$
     converges with rate of Cauchy metastability $\Phi$ implies 
     $$\sum_{k=0}^n \frac{X_k}{a_k}$$
     converges with rate of Cauchy metastability $\Delta_{\Phi}$, by Theorem \ref{thrm:quant2a}. So, the result follows from Corollary \ref{cor:qauntprobkron}.
 \end{proof}
\section{Proof-theoretic transfer principles}
\label{subsec:trans}
In collaboration with Pischke, we developed a formal system, $\mathcal{F}^\omega[\PP]$, that allows for classical analysis on a content space (that is, a finitely additive measure space) $(\Omega, S, \PP)$ \cite{NeriPischke2023}. As with most formal systems used in proof mining, ours was formulated as an extension of $\mathcal{A}^\omega:=\WEPAomega + \DC + \QFAC$, where $\WEPAomega$ denotes Peano arithmetic in all finite types, which is an extension of many-sorted classical logic with variables ranging over the set of types $T$ defined recursively as 
\[
0\in T,\quad \rho,\tau\in T\rightarrow \rho(\tau)\in T.
\]
We use the usual conversion of denoting pure types by natural numbers, that is, setting $n+1:=0(n)$. $\DC$ denotes the axiom schema of dependent choice (which implies countable choice and, thus, full comprehension over the natural numbers), and $\QFAC$ is the axiom schema of quantifier-free choice. It is well known that large parts of classical analysis can be performed in $\mathcal{A}^\omega$, and we turn the reader to \cite{kohlenbach:08:book} for full details. 

We do not spell out all the features of $\mathcal{F}^\omega[\PP]$, but note that it extends $\mathcal{A}^\omega$ with abstract types $S$ and $\Omega$, representing the algebra and sample space of a content space, along with constants $\mathrm{eq}, \in, \cup,(\cdot)^c, \emptyset, c_\Omega, \PP$ and corresponding axioms expressing that $\mathrm{eq}$ is equality on $\Omega$, $\in$ is the element relation between objects in $\Omega$ and $S$ (both defined as characteristic functions), $c_\Omega$ is an arbitrary element of $\Omega$ (so we have a concrete object expressing the nonemptyness of $\Omega$) and $\cup,(\cdot)^c, \emptyset$ being the usual operations on the algebra $S$.

In \cite{NeriPischke2023}, extensions of $\mathcal{F}^\omega[\PP]$ are presented that allow for the treatment of integration on $(\Omega, S, \PP)$ and measurable functions. Furthermore, through the introduction of a novel extension of Bezem's \cite{bezem1985strongly} majorizability, a programme extraction theorem (metatheorem) is proven that guarantees the success of the extraction of quantitative data, independent of the content space $(\Omega, S, \PP)$ for a class of generalised $\forall\exists$-statements. This metatheorem explained the success, uniformities and complexity of the bounds extracted in the seminal paper on proof mining in probability theory \cite{Avigad-Dean-Rute:Dominated:12}.

The part of \cite{NeriPischke2023} that we are concerned with here is the transfer result given in the last section, which guarantees the lifting of quantitative data of deterministic results to quantitative data from the probabilistic analogue. This result is proven formally in the systems of \cite{NeriPischke2023}, but we give an informal description of this result here to demonstrate its relation to Kronecker's lemma easily. Furthermore, all the random variables discussed in this section shall be assumed to be real-valued.

For a sequence of real numbers $\seq{x_n}$ consider the two $\Pi_3$-formulas 
\[
P(\seq{x_n})=\forall a^0\exists b^0\forall c^0 P_0(a,b,c,\seq{x_n})
\]
and
\[
Q(\seq{x_n})=\forall u^0\exists v^0\forall w^0 Q_0(u,v,w,\seq{x_n})
\]
where $P_0$ and $Q_0$ are quantifier-free. Formally, $\seq{x_n}$ is represented by the variable $x^{1(0)}$, and we use the standard encoding of real numbers as fast converging Cauchy convergence as well as the operation $\hat{\cdot}$ that transforms an object of type $1$ into a real number to express that $P_0$ and $Q_0$ are formulas on sequences of real numbers (see \cite{kohlenbach:08:book} for details).

Now, given a sequence of real valued random variables $\seq{X_n}$, we say $\seq{X_n}$ satisfies $P$ almost uniformly, and write $P(\seq{X_n})$ $\mathrm{a.u}$, if 
\[
\forall k^0,a^0\exists b^0\forall c^0\left(\PP\left(\neg P_0(a,b,c,\seq{X_n})\right)\le 2^{-k}\right).
\]
$Q(\seq{X_n})$ $\mathrm{a.u}$ is defined similarly.

For the definition of $P(\seq{X_n})$ $\mathrm{a.u}$ to make sense, the set $\{\omega \in \Omega\,|\, P_0(a,b,c,\seq{X_n(\omega)})\}$ must be measurable, for all $a,b,c \in \NN$. This is the case when $P(\seq{X_n})$ denotes Cauchy convergence (as in (\ref{eqn:newCauchy})), as $\seq{X_n}$ is a sequence of random variables (that is measurable functions) and in this case $P(\seq{X_n})$ $\mathrm{a.u}$ denotes almost uniform convergence (as presented in (\ref{eqn:a.u})).

The transfer result of \cite{NeriPischke2023} presented a strategy for one to obtain quantitative data from a result of the for 
\begin{equation*}
    P(\seq{X_n}) \mbox{ $\mathrm{a.u}$ } \to Q(\seq{X_n}) \mbox{ $\mathrm{a.u}$, }
\end{equation*}
for all sequences of bounded random variables $\seq{X_n}$, given quantitative data for the deterministic version of this result, that is, the statement for all sequences of real numbers $\seq{x_n}$,
\begin{equation*}
    P(\seq{x_n}) \to Q(\seq{x_n}),
\end{equation*}
that arises in a particularly nice form. Furthermore, such a strategy is shown to formalise in the systems introduced in \cite{NeriPischke2023}.

For a sequence of natural numbers $\seq{\tau_n}$ and a sequence of real numbers $\seq{x_n}$, we define the majorizability relation $\seq{\tau_n} \gtrsim \seq{x_n}$ by,
\begin{equation*}
    \seq{\tau_n} \gtrsim \seq{x_n} :\equiv \forall n \in \NN\,(\tau_{n+1}\ge \tau_n \land \tau_n\ge |x_n|)
\end{equation*}
The transfer result of \cite{NeriPischke2023} states:
\begin{theorem}[Theorem 10.2 of \cite{NeriPischke2023}]
    \label{thrm:trans}
Given functionals $V,A,C$ such that for all sequences of real numbers $\seq{x_n}$
\begin{equation*}
    \begin{aligned}
        \forall \underbrace{\seq{\tau_n},B,u,w}_{\omega}( \seq{\tau_n} \gtrsim \seq{x_n}\land P_0(A(\omega),B(A(\omega)),C(\omega),\seq{x_n})\\
        \to Q_0(u,V(\seq{\tau_n},B,u),w,\seq{x_n})),
    \end{aligned}
\end{equation*}
(where we quantify over all sequences of natural numbers $\seq{\tau_n}$) then for all sequences of bounded random variables $\seq{X_n}$, we can construct $V',A',C'$ such that
\begin{equation*}
    \begin{aligned}
        \forall \underbrace{B,k,u,w}_{\alpha}( \PP(\neg P_0(A'(\alpha),B(k,A'(\alpha)),C'(\alpha), \seq{X_n}))\leq 2^{-k}\\
        \to \PP(\neg Q_0(u,V'(B,k,u),w, \seq{X_n}))\leq 2^{-k})
    \end{aligned}
\end{equation*}
Furthermore, the functionals $V',A',C'$ can be constructed from the proof, and depending on $V, A,C$ and also on  $\seq{X_n}$, via a nondecreasing sequence of natural numbers $\seq{Z_n}$ witnessing the boundedness of $\seq{X_n}$, that is, satisfying for every $n \in \NN$ and  $\omega \in \Omega$, we have  $Z_n \ge |X_n(\omega)|$. 
\end{theorem}
The premise of Theorem \ref{thrm:trans} represents a solution to the Dialectica interpretation of 
\begin{equation}
\label{eqn:detimp}
    \forall \seq{x_n}\,(P(\seq{x_n}) \to Q(\seq{x_n})).
\end{equation}
formalised suitably in $\mathcal{F}^\omega[\PP]$. Furthermore, the functionals $V, A, C$ represent realisers of essentially the Dilectica interpretation of (\ref{eqn:detimp}), depending on $\seq{x_n}$ only via a bounding sequence and with additional uniformities.

If such a solution to the Dialectica interpretation can be found for (\ref{eqn:detimp}), then the conclusion of the transfer result solves the Dialectica interpretation of
\begin{equation*}
    P(\seq{X_n}) \mbox{ $\mathrm{a.u}$ } \to Q(\seq{X_n}) \mbox{ $\mathrm{a.u}$, }
\end{equation*}
for bounded sequences of random variables $\seq{X_n}$. Such a solution can then be easily used to obtain a quantitative result, where form $P(\seq{X_n}) \mbox{ $\mathrm{a.u}$}$ can be converted to rates for $Q(\seq{X_n}) \mbox{ $\mathrm{a.u}$.}$
 
If we take  
$$P_0(a,b,c,\seq{x_n}) := \forall m,n \in [b;c] \left|\sum_{k=0}^n x_k - \sum_{k=0}^m x_k \right| \le 2^{-a}$$
 and 
$$Q_0(u,v,w,\seq{x_n}) := \forall n \in [v;w] \left|\frac{1}{a_n}\sum_{k=0}^n a_kx_k \right| \le 2^{-u}$$

with the sufficient assumptions on $\seq{a_n}$, Kronecker's lemma becomes:
\begin{equation*}
    \forall \seq{x_n}(P(\seq{x_n}) \to Q(\seq{X_n}).
\end{equation*}
and the probabilistic Kronecker's lemma is 
\begin{equation*}
    P(\seq{X_n}) \mbox{ $\mathrm{a.u}$ } \to Q(\seq{X_n}) \mbox{ $\mathrm{a.u}$ }.
\end{equation*}
Thus, from the proof of Theorem \ref{thrm:quantkron}, taking
\begin{equation*}
\begin{aligned}
    &A(\seq{\tau_n}, B,u,w):= u+2\\
    &C(\seq{\tau_n}, B,u,w):= w\\
    &V(\seq{\tau_n}, B,u):= \max\left\{B(u+2),f_{\seq{a_n}}\left(2^{u+2}a_{B(u+2)}\sum_{i=0}^{B(u+2)}\tau_i\right) \right\}
\end{aligned}
\end{equation*}
with $f_{\seq{a_n}}$ as defined in Theorem \ref {thrm:quantkron}, allows the premise for Theorem \ref{thrm:trans} to be satisfied. The proof of Theorem \ref{thrm:trans} given in \cite{NeriPischke2023} can be used to obtain a quantitative version of the probabilistic Kronecker's lemma for bounded random variables. However, in the computational interpretation of Kronecker's lemma we gave in Theorem \ref{thrm:quantprobkron1}, we did not assume that the random variables were bounded, only that they were almost surely finite. The reason we were able to do this was due to the fact that the computational solution to Kronecker's lemma we obtained contained further uniformities; that is, the functional $A, C$ and $V$ are uniformly continuous in their first argument and, in particular, all have a modulus of continuity given by $M(B,u,w):= B(u+2)$, that is for all sequences $\seq{\tau_n^1}$ and $\seq{\tau_n^2}$, $B:\NN\to \NN$ and $u,w \in \NN$ if we have $\tau_n^1 = \tau_n^2$ for all $n \le M(B,u,w)$ then $V(\seq{\tau_n^1}, B,u) = V(\seq{\tau_n^2}, B,u)$ with the same holding for $A$ and $C$. This arises in Kronecker's lemma because the functionals $A$ and $C$ are independent of $\seq{\tau_n}$ and thus are trivially uniformly continuous. Although such a uniform solution is not guaranteed to exist in general, this tends to be true when one obtains quantitative results on the implications between two convergence statements. It is typical for $A$ and $C$ not to depend on the sequence. Furthermore, when the rate for the conclusion (in terms of the rate for the premise) does depend on the sequence, it only depends on a finite initial segment of the sequence, thus guaranteeing the uniform continuity requirement (see \cite[Remark 3.5]{kohlenbach2020rates} and \cite[Remark 3.2]{powell2021rates} for recent examples of this phenomenon in analysis.)

For such a solution, we have the following general transfer principle, which generalises the proof of Theorem \ref{thrm:quantprobkron1}:
\begin{theorem}
    \label{thrm:transnew}
Suppose we have functionals $V,A,C$ such that for all sequences of real numbers $\seq{x_n}$
\[
\forall \underbrace{\seq{\tau_n},B,u,w}_{\omega}\left( \seq{\tau_n} \gtrsim \seq{x_n}\land P_0(A(\omega),B(A(\omega)),C(\omega),\seq{x_n})\to Q_0(u,V(\seq{\tau_n},B,u),w,\seq{x_n})\right),
\]
and $A, C$ and $V$ are uniformly continuous in their first argument, each with a modulus of continuity $M(B,u,w)$. Then for all sequences of random variables $\seq{X_n}$ with a function $Z:\NN \times \NN \to \NN$ satisfying 
\begin{equation}
\label{eqn:abscondition}
    \forall k,p \in \NN\, \left(\PP\left(\bigcup_{i=0}^p\{ |X_i| > Z(k,p)\}\right) \le 2^{-k}\right)
\end{equation}

we can construct $V',A',C'$ in terms of $A,B,C,Z,M$ such that
\begin{equation*}
    \begin{aligned}
        \forall \underbrace{B,k,u,w}_{\alpha}( \PP(\neg P_0(A'(\alpha),B(k,A'(\alpha)),C'(\alpha), \seq{X_n}))\leq 2^{-(k+1)}\\ 
        \to \PP(\neg Q_0(u,V'(B,k,u),w, \seq{X_n}))\leq 2^{-k}).
    \end{aligned}
\end{equation*}
\end{theorem}
\begin{proof}
Take $V,A,C$ and $M$ satisfying the premise of the theorem and $\alpha=(B,k,u,w)$. Define 
\begin{gather*}
A'(\alpha):=A(\seq{z_n},B(k),u,w),\\
C'(\alpha):=C(\seq{z_n},B(k),u,w),\\
V'(B,k,u):=V(\seq{z_n},B(k),u).
\end{gather*}
With $z_n:= Z(k+1,M(B,u,w))$ for all $n \in \NN$. We have 
\begin{equation*}
\begin{aligned}
    &\PP(\neg Q_0(u,V'(B,k,u),w, \seq{X_n}))\\
    &\leq \PP\left(\neg Q_0(u,V'(B,k,u),w, \seq{X_n})\wedge\bigwedge_{i=0}^{M(B,u,w)}\{ |X_i| > Z(k+1,M(B,u,w))\}\right)\\
    &+ \PP\left(\neg Q_0(u,V'(B,k,u),w, \seq{X_n})\wedge\bigwedge_{i=0}^{M(B,u,w)}\{ |X_i| \le Z(k+1,M(B,u,w))\}\right)\\
    &\le \PP\left(\neg Q_0(u,V'(B,k,u),w, \seq{X_n})\wedge\bigwedge_{i=0}^{M(B,u,w)}\{ |X_i| \le Z(k+1,M(B,u,w))\}\right)\\
    &+ 2^{-(k+1)}. 
\end{aligned}
\end{equation*}
Now, take $\delta \in \Omega$ satisfying 
\begin{equation*}
    \neg Q_0(u,V'(B,k,u),w, \seq{X_n(\delta)})\land\bigwedge_{i=0}^{M(B,u,w)}( |X_i(\delta)| \le Z(k+1,M(B,u,w))).
\end{equation*}
For such a $\delta$, define $\seq{\tau_n(\delta)}$ to be an arbitrary increasing sequence of natural numbers such that as $\tau_n:=Z(k+1,M(B,u,w))$ for $n \le M(B,u)$ and $X_n(\delta) \le \tau_n(\delta)$ for all $n$. Therefore, we have $\seq{\tau_n(\delta)}\gtrsim \seq{X_n(\delta)}$. Therefore, by unwinding the definition of $V'$ and using the fact that for all $n\le M(B,u,w)$, $\tau_n(\delta)=Z_n$ so, by the continuity of $V$ we have $V(\seq{z_n},B(k),u) = V(\seq{\tau_n(\delta)},B(k),u)$ we must have 
\begin{equation*}
    \neg Q_0(u,V(\seq{\tau_n(\delta)},B(k),u),w, \seq{X_n(\delta)})
\end{equation*}
which implies
\begin{equation*}
    \neg P_0(A(\seq{\tau_n(\delta)},B(k),u,w),B(k,A(\seq{\tau_n(\delta)},B(k),u,w)),C(\seq{\tau_n(\delta)},B(k),u,w),\seq{X_n(\delta)}).
\end{equation*}
Now, by the continuity of $A$ and $C$, and the definition of $A'$ and $C'$ we must have,
\begin{equation*}
    \neg P_0(A'(\alpha),B(k,A'(\alpha)),C'(\alpha),\seq{X_n(\delta)}).
\end{equation*}
Therefore,
\begin{equation*}
\begin{aligned}
        \neg Q_0(u,V'(B,k,u),w, \seq{X_n})\wedge\bigwedge_{i=0}^{M(B,u)}\{ |X_i| \le Z(k+1,M(B,u))\}\\
        \to \neg P_0(A'(\alpha),B(k,A'(\alpha)),C'(\alpha),\seq{X_n})
\end{aligned}
\end{equation*}
which implies, 
\begin{equation*}
    \begin{aligned}
        \PP\left(\neg Q_0(u,V'(B,k,u),w, \seq{X_n})\wedge\bigwedge_{i=0}^{M(B,u)}\{ |X_i| \le Z(k+1,M(B,u))\}\right)\\
        \le \PP(\neg P_0(A'(\alpha),B(k,A'(\alpha)),C'(\alpha), \seq{X_n}))\leq 2^{-(k+1)}
    \end{aligned}
\end{equation*}
so we are done.
\end{proof}
\begin{remark}
Theorem \ref{thrm:transnew} is indeed formalisable in the systems of \cite{NeriPischke2023}. However, random variables in \cite{NeriPischke2023} are treated as objects of type $1(\Omega)$, and a constant majorises such objects in these systems. Therefore, the metatheorem \cite[Theorem 8.9]{NeriPischke2023} only guarantees bound extraction for theorems regarding bounded random variables. If Theorem \ref{thrm:transnew} were to be used alongside the metatheorem of \cite{NeriPischke2023}, one would have to assume that the random variables are bounded, resulting in no improvement over Theorem \ref{thrm:trans}. 
\end{remark}
\section{Computability and the strong law of large numbers}
\label{sec:computability}
We saw in Section \ref{subsec:kroncomp} that computable rates for the conclusion of Kronecker's lemma must depend on computable rates from its premise. In this short section, we show that a similar phenomenon occurs in the Strong Law of Large Numbers. 

\begin{example}
\label{ex:comp}
Let us take a recursively enumerable set, $A$, that is not recursive. Let $\seq{a_n}$ be a recursive enumeration of the elements in $A$.

Let $\seq{X_n}$ be an independent sequence of discrete random variables, with distributions given by,
\begin{equation*}
\mathbb{P}(X_n = x):=\begin{cases}
	2^{-a_n-1} & \parbox[t]{.6\textwidth}{\rm if $x = n - n2^{-a_n-1}$}\\[0.5cm]
	1-2^{-a_n-1}  & \mbox{\rm if $x = - n2^{-a_n-1}$}\\[0.5cm]
        0  & \mbox{\rm o/w.}
\end{cases}
\end{equation*}
Then, one can easily see that,
\begin{equation*}
   \EE(X_n) = 0,\mbox{ and  } \sum_{n=1}^\infty \frac{\mathrm{Var}(X_n)}{n^2} \le \frac{5}{12}.
\end{equation*}
However, there is no computable function $\phi: \QQ^+ \times \QQ^+ \to \NN$ such that
\begin{equation*}
    \forall \varepsilon, \lambda \in \QQ^+\, \forall n \ge \phi(\varepsilon,\lambda)\, \left(\PP\left(\sup_{m \ge n} \left|\frac{S_m}{m}\right| > \varepsilon\right) \le \lambda\right).
\end{equation*}
Let,
$$x_n := \sum_{k=1}^n k2^{-a_k-1}$$ 
and observe that we can write $S_n = K_n - x_n$, with $K_n \in \NN.$

Suppose there is such a computable function $\phi$, such that for all $\varepsilon, \lambda \in \QQ^+$ and $ n \ge \phi(\varepsilon,\lambda)$
\begin{equation*}
     \mathbb{P}\left(\max_{m \ge n} \left|\frac{1}{m}S_{m}\right| > \varepsilon\right) \le \lambda.
\end{equation*}
This is equivalent to 
\begin{equation*}
     \mathbb{P}\left(\forall m \ge n\, \left|\frac{1}{m}S_{m}\right| \le \varepsilon\right) > 1- \lambda, 
\end{equation*}
which is equivalent to
\begin{equation}
     \label{eqn:Prob}
     \mathbb{P}\left(\forall m \ge n\,\left(-\varepsilon + \frac{1}{m}x_{m} \le \frac{1}{m}K_{m} \le \varepsilon + \frac{1}{m}x_{m}\right)\right) > 1- \lambda. 
\end{equation}
We now describe an effective procedure to determine whether $k \in \NN$ is in $A$, which will contradict the assumption that $A$ is not a recursive set, leading us to the conclusion that no computable function $\phi$ can exist. Suppose $M \ge \phi(\frac{1}{2},2^{-k-1})$. We have, from (\ref{eqn:Prob}),
\begin{equation*}
     \mathbb{P}\left(\forall m \ge \phi\left(\frac{1}{2},2^{-k-1}\right)\,\left(-\frac{1}{2} + \frac{1}{m}x_{m} \le \frac{1}{m}K_{m} \le \frac{1}{2} + \frac{1}{m}x_{m}\right)\right) > 1- 2^{-k-1}.
\end{equation*}
Now observe 
\begin{equation}
    \label{eqn: bound}
    \frac{1}{M}x_{M} = \frac{1}{M}\sum_{k=1}^M k2^{-a_k-1} < \sum_{k=1}^M 2^{-a_k-1} < \frac{1}{2}.
\end{equation}
We have that,
$$\forall m \ge \phi\left(\frac{1}{2},2^{-k-1}\right)\,\left(-\frac{1}{2} + \frac{1}{m}x_{m} \le \frac{1}{m}K_{m} \le \frac{1}{2} + \frac{1}{m}x_{m}\right)$$
implies  $$-\frac{1}{2} + \frac{1}{M}x_{M} \le \frac{1}{M}K_{M} \le \frac{1}{2}+ \frac{1}{M}x_{M}.$$
This further implies $\frac{1}{M}K_{M} < 1$ by (\ref{eqn: bound}), which implies $X_M  = - M2^{-a_M-1}$. Thus we have, 
\begin{equation*}
    \begin{aligned}
    &\PP(X_M = - M2^{-a_M-1})\\
    &\qquad\qquad\ge\mathbb{P}\left(\forall m \ge \phi\left(\frac{1}{2},2^{-k-1}\right)\,\left(-\frac{1}{2} + \frac{1}{m}x_{m} \le \frac{1}{m}K_{m} \le \frac{1}{2} + \frac{1}{m}x_{m}\right)\right)\\
    &\qquad\qquad>1- 2^{-k-1}.
    \end{aligned}
\end{equation*}
So, $1-2^{-a_M-1} > 1- 2^{-k-1}$ which implies $a_M > k$. Thus if $M \ge \phi(\frac{1}{2},2^{-k-2})$ then $k \neq a_M$. Thus to effectively determine if $k \in A$, it suffices to check if $k = a_m$ for $m < \phi(\frac{1}{2},2^{-k-1})$ effectively, which can be done.
\end{example}

This example also demonstrates the ineffectiveness of Chung's strong law of large numbers, as it is a generalisation of Kolmogorov's.

\section{Concluding remarks}
\label{sec:transfer}
We would like to mention our ongoing work on establishing rates for Theorem \ref{thrm:SLLN}. Both Kolmogorov's original proof and Etemadi's elementary proof (extending this result to pairwise independent random variables \cite{etemadi1981elementary}) utilise the distributions of the random variables in crucial ways. We anticipate the development of a construction similar to Example \ref{thrm:kroncomp}, demonstrating that a direct computable rate of almost sure convergence must rely on the distribution of the random variables. Some progress in this direction has been made in \cite{gacs2010constructive}; however, these results are obtained with additional structural assumptions on the underlying sample space. Furthermore, we conjecture that metastable rates of almost sure convergence can be that are independent of the distribution of the random variables.

Lastly, we note that by exploiting the monotonicity properties of the formula expressing Cauchy convergence and using the continuity of the probability measure, one can show that (\ref{eqn:a.u}) is equivalent to 
\begin{equation}
\label{eqn:infa.u}
  \forall \varepsilon, \lambda > 0\, \exists N \in \NN\,\left(\mathbb{P}\left(\sup_{ n \ge N}\norm{Y_n -  Y_N} > \varepsilon\right)\le \lambda\right).
\end{equation}
and given $\varepsilon, \lambda$ we can take the same $N$ realising both definition. The above definition tends to be taken in the probability literature (see, \cite{siegmund1975large,luzia_2018} for example). The reason we opt for (\ref{eqn:a.u}) is to demonstrate that all of our results can, in principle, be formalised in (a suitable extension) of the systems in \cite{NeriPischke2023}. The formal system fully axiomatises the theory of contents (finitely additive probability spaces), and one can only reason about infinite unions in a limited way. All of the results in this paper did not rely on infinite unions.\\

{\bf Acknowledgments:} This article was written as part of the author’s PhD studies under the supervision of Thomas Powell, and I would like to thank him for his support and invaluable guidance. The author would also like to thank Nicholas Pischke for insightful comments and discussions. The author was partially supported by the EPSRC Centre
for Doctoral Training in Digital Entertainment (EP/L016540/1).

\bibliographystyle{acm}
\bibliography{ref}

\begin{thebibliography}{10}

\bibitem{arthan2021borel}
{\sc Arthan, R., and Oliva, P.}
\newblock {On the Borel-Cantelli Lemmas, the Erd{\H{o}}s-R{\'e}nyi Theorem, and the Kochen-Stone Theorem}.
\newblock {\em Journal of Logic and Analysis 13\/} (2021).

\bibitem{Avigad-Dean-Rute:Dominated:12}
{\sc Avigad, J., Dean, E., and Rute, J.}
\newblock A metastable dominated convergence theorem.
\newblock {\em Journal of Logic and Analysis\/} (2012).

\bibitem{AVIGAD-GERHARDY-TOWSNER:10:Ergodic}
{\sc Avigad, J., Gerhardy, P., and Towsner, H.}
\newblock Local stability of ergodic averages.
\newblock {\em Transactions of the American Mathematical Society 362}, 1 (2010), 261--288.

\bibitem{bezem1985strongly}
{\sc Bezem, M.}
\newblock Strongly majorizable functionals of finite type: A model for barrecursion containing discontinuous functionals.
\newblock {\em The Journal of Symbolic Logic 50}, 3 (1985), 652--660.

\bibitem{chung:47:LLN}
{\sc Chung, K.}
\newblock {Note on Some Strong Laws of Large Numbers}.
\newblock {\em American Journal of Mathematics 69}, 1 (1947), 189--192.

\bibitem{delporte1964functions}
{\sc Delporte, J.}
\newblock Almost surely continuous random functions on a closed interval.
\newblock In {\em Annals of the IHP Probabilities and Statistics\/} (1964), pp.~111--215.

\bibitem{siegmund1975large}
{\sc D.Siegmund}.
\newblock Large deviation probabilities in the strong law of large numbers.
\newblock {\em Zeitschrift f{\"u}r Wahrscheinlichkeitstheorie und Verwandte Gebiete 31}, 2 (1975), 107--113.

\bibitem{etemadi1981elementary}
{\sc Etemadi, N.}
\newblock An elementary proof of the strong law of large numbers.
\newblock {\em Zeitschrift f{\"u}r Wahrscheinlichkeitstheorie und verwandte Gebiete 55}, 1 (1981), 119--122.

\bibitem{freund2022re}
{\sc Freund, A., and Kohlenbach, U.}
\newblock {R.E. Bruck, proof mining and a rate of asymptotic regularity for ergodic averages in Banach spaces}.
\newblock {\em Applied Set-Valued Analysis \& Optimization 4}, 3 (2022).

\bibitem{freund2023bounds}
{\sc Freund, A., and Kohlenbach, U.}
\newblock {Bounds for a nonlinear ergodic theorem for Banach spaces}.
\newblock {\em Ergodic Theory and Dynamical Systems 43}, 5 (2023), 1570--1593.

\bibitem{gacs2010constructive}
{\sc G{\'a}cs, P.}
\newblock {Constructive law of large numbers with application to countable Markov chains}.
\newblock {\em Boston university\/} (2010).

\bibitem{gerhardy2008general}
{\sc Gerhardy, P., and Kohlenbach, U.}
\newblock General logical metatheorems for functional analysis.
\newblock {\em Transactions of the American Mathematical Society 360}, 5 (2008), 2615--2660.

\bibitem{godel1958bisher}
{\sc G{\"o}del, K.}
\newblock {{\"U}ber eine bisher noch nicht ben{\"u}tzte Erweiterung des finiten Standpunktes}.
\newblock {\em dialectica 12}, 3-4 (1958), 280--287.

\bibitem{gut2005probability}
{\sc Gut, A.}
\newblock {\em Probability: a graduate course}.
\newblock Springer, 2013.

\bibitem{hoffmann1976law}
{\sc Hoffmann-J{\o}rgensen, J., and Pisier, G.}
\newblock {The law of large numbers and the central limit theorem in Banach spaces}.
\newblock {\em The Annals of Probability\/} (1976), 587--599.

\bibitem{kachurovskii2019measuring}
{\sc Kachurovskii, A., and Podvigin, I.}
\newblock {Measuring the rate of convergence in the Birkhoff ergodic theorem}.
\newblock {\em Mathematical Notes 106\/} (2019), 52--62.

\bibitem{kohlenbach2005some}
{\sc Kohlenbach, U.}
\newblock Some logical metatheorems with applications in functional analysis.
\newblock {\em Transactions of the American Mathematical Society 357}, 1 (2005), 89--128.

\bibitem{kohlenbach:08:book}
{\sc Kohlenbach, U.}
\newblock {\em {Applied Proof Theory: Proof Interpretations and their Use in Mathematics}}.
\newblock Springer Monographs in Mathematics. Springer, 2008.

\bibitem{kohlenbach2020rates}
{\sc Kohlenbach, U., and Powell, T.}
\newblock Rates of convergence for iterative solutions of equations involving set-valued accretive operators.
\newblock {\em Computers \& Mathematics with Applications 80}, 3 (2020), 490--503.

\bibitem{Kol1930}
{\sc Kolmogorov, A.}
\newblock {Sur la loi forte des grands nombres}.
\newblock {\em Comptes rendus de l'Acad{\'e}mie des Sciences 191\/} (1930), 910--912.

\bibitem{korchevsky2018rate}
{\sc Korchevsky, A.}
\newblock {On the Rate of Convergence in the Strong Law of Large Numbers for Nonnegative Random Variables}.
\newblock {\em Journal of Mathematical Sciences 229}, 6 (2018), 719--727.

\bibitem{kreisel:51:proofinterpretation:part1}
{\sc Kreisel, G.}
\newblock {On the Interpretation of Non-Finitist Proofs, {P}art {I}}.
\newblock {\em Journal of Symbolic Logic 16\/} (1951), 241--267.

\bibitem{kreisel:52:proofinterpretation:part2}
{\sc Kreisel, G.}
\newblock {On the Interpretation of Non-Finitist Proofs, {P}art {II}: Interpretation of Number Theory}.
\newblock {\em Journal of Symbolic Logic 17\/} (1952), 43--58.

\bibitem{ledoux2013probability}
{\sc Ledoux, M., and Talagrand, M.}
\newblock {\em {Probability in Banach Spaces: isoperimetry and processes}}.
\newblock Springer Science \& Business Media, 2013.

\bibitem{li2017introduction}
{\sc Li, D., and Queff{\'e}lec, H.}
\newblock {\em {Introduction to Banach Spaces: Analysis and Probability: Volume 1}}, vol.~166.
\newblock Cambridge University Press, 2017.

\bibitem{luzia_2018}
{\sc Luzia, N.}
\newblock A simple proof of the strong law of large numbers with rates.
\newblock {\em Bulletin of the Australian Mathematical Society 97}, 3 (2018), 513–517.

\bibitem{Masani76}
{\sc Masani, P.}
\newblock {Measurability and Pettis integration in Hilbert spaces}.
\newblock In {\em Measure Theory\/} (Berlin, Heidelberg, 1976), Springer Berlin Heidelberg, pp.~69--105.

\bibitem{Ner2023:LLN}
{\sc Neri, M.}
\newblock {Quantitative Strong Laws of Large Numbers}.
\newblock Preprint, available at \url{https://arxiv.org/abs/2406.19166}, 2024.

\bibitem{NeriPischke2023}
{\sc Neri, M., and Pischke, N.}
\newblock Proof mining and probability theory.
\newblock Preprint, available at \url{https://arxiv.org/abs/2403.00659}, 2024.

\bibitem{NeriPowell2023}
{\sc Neri, M., and Powell, T.}
\newblock A computational study of a class of recursive inequalities.
\newblock {\em Journal of Logic and Analysis 15}, 3 (2023), 1--48.

\bibitem{NeriPowell:RS:2024}
{\sc Neri, M., and Powell, T.}
\newblock {A quantitative Robbins-Siegmund theorem}.
\newblock Preprint, available at \url{https://arxiv.org/abs/2410.15986}, 2024.

\bibitem{NeriPowell:martingales:2024}
{\sc Neri, M., and Powell, T.}
\newblock On quantitative convergence for stochastic processes: Crossings, fluctuations and martingales.
\newblock Preprint, available at \url{https://arxiv.org/abs/2406.19979}, 2024.

\bibitem{PischkePowell:Halpern:2024}
{\sc Pischke, N., and Powell, T.}
\newblock {Asymptotic regularity of a generalised stochastic Halpern scheme with applications}.
\newblock Preprint, available at \url{https://arxiv.org/abs/2411.04845}, 2024.

\bibitem{powell2020note}
{\sc Powell, T.}
\newblock {A note on the finitization of Abelian and Tauberian theorems}.
\newblock {\em Mathematical Logic Quarterly 66}, 3 (2020), 300--310.

\bibitem{powell2021rates}
{\sc Powell, T., and Wiesnet, F.}
\newblock Rates of convergence for asymptotically weakly contractive mappings in normed spaces.
\newblock {\em Numerical Functional Analysis and Optimization 42}, 15 (2021), 1802--1838.

\bibitem{robbins1971convergence}
{\sc Robbins, H., and Siegmund, D.}
\newblock {A convergence theorem for non negative almost supermartingales and some applications}.
\newblock In {\em Optimizing methods in statistics}. Elsevier, 1971, pp.~233--257.

\bibitem{Seneta13}
{\sc Seneta, E.}
\newblock {A Tricentenary history of the Law of Large Numbers}.
\newblock {\em Bernoulli 19}, 4 (2013), 1088 -- 1121.

\bibitem{simpson2009subsystems}
{\sc Simpson, S.~G.}
\newblock {\em Subsystems of second order arithmetic}, vol.~1.
\newblock Cambridge University Press, 2009.

\bibitem{specker:49:sequence}
{\sc Specker, E.}
\newblock {{N}icht Konstruktiv Beweisbare {S}{\"a}tze der {A}nalysis}.
\newblock {\em Journal of Symbolic Logic 14\/} (1949), 145--158.

\bibitem{tao:07:softanalysis}
{\sc Tao, T.}
\newblock Soft analysis, hard analysis, and the finite convergence principle.
\newblock Essay posted 23 May 2007, 2007.
\newblock Appeared in: ‘T. Tao, Structure and Randomness: Pages from Year One of a Mathematical Blog. AMS, 298pp., 2008’.

\bibitem{tao:08:ergodic}
{\sc Tao, T.}
\newblock Norm convergence of multiple ergodic averages for commuting transformations.
\newblock {\em Ergodic Theory and Dynamical Systems 28\/} (2008), 657--688.

\bibitem{powell2023finitization}
{\sc T.Powell}.
\newblock {A finitization of Littlewood's Tauberian theorem and an application in Tauberian remainder theory}.
\newblock {\em Annals of Pure and Applied Logic 174}, 4 (2023), 103231.

\bibitem{woyczynski1974random}
{\sc Woyczynski, W.}
\newblock {Random series and laws of large numbers in some Banach spaces}.
\newblock {\em Theory of Probability \& Its Applications 18}, 2 (1974), 350--355.

\end{thebibliography}
\end{document}